\documentclass[a4paper]{amsart}
\usepackage{bbm}
\usepackage{color}
\usepackage{enumerate}
\usepackage{cite}
\usepackage[utf8]{inputenc}
\usepackage[all,cmtip]{xy}
\usepackage{etoolbox}
\apptocmd{\thebibliography}{\raggedright}{}{}

\usepackage{amscd}

\usepackage{extarrows}

\usepackage{tikz}
\usetikzlibrary{arrows}

\usepackage[leqno]{amsmath}
\usepackage{amsthm,amssymb,amsfonts}

\theoremstyle{plain}
\newtheorem{Theorem}{Theorem}[section]
\makeatletter
\newlength{\@thlabel@width}%
\newcommand{\thmenumhspace}{\settowidth{\@thlabel@width}{\itshape1.}\sbox{\@labels}{\unhbox\@labels\hspace{\dimexpr-\leftmargin+\labelsep+\@thlabel@width-\itemindent}}}
\makeatother
\newtheorem{Lemma}[Theorem]{Lemma}

\newtheorem{Corollary}[Theorem]{Corollary}
\newtheorem{Proposition}[Theorem]{Proposition}

\theoremstyle{definition}
\newtheorem{Def}[Theorem]{Definition}
\newtheorem{Remark}[Theorem]{Remark}

\numberwithin{equation}{section}

\newcommand{\Z}{\ensuremath{\mathbb{Z}}}

\newcommand{\OO}{\ensuremath{\mathcal{O}}}
\newcommand{\Q}{\ensuremath{\mathbb{Q}}}

\newcommand{\R}{\ensuremath{\mathbb{R}}}
\newcommand{\C}{\ensuremath{\mathbb{C}}}

\newcommand{\p}{\ensuremath{\mathfrak{p}}}
\renewcommand{\P}{\ensuremath{\mathfrak{P}}}
\newcommand{\q}{\ensuremath{\mathfrak{q}}}

\newcommand{\A}{\ensuremath{\mathbb{A}}}

\newcommand{\inv}{^{-1}}									

\newcommand{\tild}[1]{\ensuremath{\widetilde{#1}}}						

\newcommand{\too}{\longrightarrow}								
\newcommand{\mapstoo}{\longmapsto}

\newcommand{\hooklongrightarrow}{\lhook\joinrel\too}
\newcommand{\into}{\hookrightarrow}
\newcommand{\onto}{\twoheadrightarrow}

\newcommand{\cf}{{\mathbbm 1}}
\newcommand{\PP}{\ensuremath{\mathbb{P}^1}}					
\newcommand{\sinfty}{\ensuremath{^{S,\infty}}}
\newcommand{\PPP}{\ensuremath{\mathbb{P}}}	

\newcommand{\G}{\ensuremath{\mathcal{G}}}
\newcommand{\n}{\ensuremath{\mathfrak{n}}}
\newcommand{\m}{\ensuremath{\mathfrak{m}}}

\newcommand{\T}{\ensuremath{\mathcal{T}}}
\newcommand{\V}{\ensuremath{\mathcal{V}}}
\newcommand{\E}{\ensuremath{\vec{\mathcal{E}}}}

\newcommand{\St}{\textnormal{St}}
\newcommand{\TSt}{\textnormal{St}^{\textnormal{tw}}}
\newcommand{\HH}{\textnormal{H}}
\newcommand{\dd}{\textnormal{d}}
\newcommand{\nr}{\textnormal{nr}}
\newcommand{\tw}{\textnormal{tw}}
\newcommand{\LI}{\mathcal{L}}
\newcommand{\dR}{\textnormal{dR}}
\DeclareMathOperator{\ES}{ES}

\DeclareMathOperator{\Hom}{Hom}

\DeclareMathOperator{\Gal}{Gal}
\DeclareMathOperator{\invol}{inv}

\DeclareMathOperator{\Ker}{ker}

\DeclareMathOperator{\id}{id}

\DeclareMathOperator{\GL}{GL}
\DeclareMathOperator{\PGL}{PGL}

\DeclareMathOperator{\rec}{rec}
\DeclareMathOperator{\drec}{drec}
\DeclareMathOperator{\Div}{Div}
\DeclareMathOperator{\Ends}{Ends_{\p}}

\DeclareMathOperator{\Dist}{\operatorname{Dist}}

\DeclareMathOperator{\Tdelta}{\delta^{tw}_\p}
\DeclareMathOperator{\cind}{c-ind}

\DeclareMathOperator{\ev}{ev}
\DeclareMathOperator{\Tev}{ev^{tw}_\p}
\DeclareMathOperator{\Ev}{Ev}

\DeclareMathOperator{\ram}{ram}

\DeclareMathOperator{\univ}{univ}

\DeclareMathOperator{\INV}{inv}
\DeclareMathOperator{\Ad}{Ad}

\DeclareMathOperator{\ord}{ord}
\DeclareMathOperator{\Stab}{Stab}

\DeclareMathOperator{\adic}{adic}
\DeclareMathOperator{\Tate}{Tate}
\DeclareMathOperator{\Jac}{Jac}

\def\Xint#1{\mathchoice%
{\XXint\displaystyle\textstyle{#1}}%
{\XXint\textstyle\scriptstyle{#1}}%
{\XXint\scriptstyle\scriptscriptstyle{#1}}%
{\XXint\scriptscriptstyle\scriptscriptstyle{#1}}%
\!\int}%
\def\XXint#1#2#3{{\setbox0=\hbox{$#1{#2#3}{\int}$}%
\vcenter{\hbox{$#2#3$}}\kern-.5\wd0}}%

\title[Anticyclotomic Stickelberger elements]{Leading terms of anticyclotomic Stickelberger elements and $p$-adic periods}
\subjclass[2010]{Primary 11F67; Secondary 11F75, 11G18, 11G40}

\author[F. Bergunde]{Felix Bergunde}
\address{F. Bergunde \\ Fakult\"at f\"ur Mathematik \\ Universit\"at Bielefeld \\ Universit\"atsstra{\ss}e 25 \\ 33615 Bielefeld \\ Germany}
\email{fbergund@math.uni-bielefeld.de}
\author[L. Gehrmann]{Lennart Gehrmann}
\address{L. Gehrmann \\ Fakult\"at f\"ur Mathematik \\ Universit\"at Duisburg-Essen \\ Thea-Leymann-Stra{\ss}e 9 \\ 45127 Essen \\ Germany}
\email{lennart.gehrmann@uni-due.de}

\begin{document}

\begin{abstract}
Let $E$ be a quadratic extension of a totally real number field.
We construct Stickelberger elements for Hilbert modular forms of parallel weight 2 in anticyclotomic extensions of $E$.
Extending methods developed by Dasgupta and Spie\ss~ from the multiplicative group to an arbitrary one-dimensional torus we bound the order of vanishing of these Stickelberger elements from below and, in the analytic rank zero situation, we give a description of their leading terms via automorphic $\LI$-invariants.
If the field $E$ is totally imaginary, we use the $p$-adic uniformization of Shimura curves to show the equality between automorphic and arithmetic $\LI$-invariants.
This generalizes a result of Bertolini and Darmon from the case that the ground field is the field of rationals to arbitrary totally real number fields.\end{abstract}

\maketitle

\tableofcontents

\section*{Introduction}
\textbf{Introduction}.

Let $A$ be an elliptic curve over the field of rational numbers.
In the seminal article \cite{MTT} Mazur, Tate and Teitelbaum formulate a $p$-adic Birch and Swinnerton-Dyer conjecture for the $p$-adic $L$-function $L_p(A,s)$ associated to $A$.
Generically, the order of vanishing of $L_p(A,s)$ at $s=0$ should be equal to the rank of $A(\Q)$.
But in the case of split multiplicative reduction at $p$ the vanishing of an Euler like factor forces $L_p(A,s)$ to vanish at $s=0$ even though the complex $L$-function $L(A,s)$ might not vanish at $s=1$.
In this situation one expects that $$\ord_{s=0}L_{p}(A,s)=  \ord_{s=1}L(A,s)+1.$$
Thus, if the complex $L$-function does not vanish at $s=1$, the $p$-adic $L$-function should have a simple zero.
In \cite{MTT} the following leading term formula is proposed:
\begin{align}\label{formula}
 L_{p}^{\prime}(A,0)=\LI(A)\ \frac{L(A,1)}{\Omega_{A}}
\end{align}
Here $\Omega_{A}$ is a real period attached to $A$ and $\LI(A)$ is the $\LI$-invariant of $A$ at $p$, which is defined as follows:
Since $A$ has split multiplicative reduction there exists a rigid analytic uniformization
$$\mathbb{G}_m/\langle q_p^{\Tate}\rangle\too A_{\Q_p}$$
with $q_p^{\Tate}\in \Q_{p}^{\ast}$ the Tate period of $A_{\Q_p}$.
The $\LI$-invariant is given by the quotient $$\LI(A)=\frac{\log_p(q_p^{\Tate})}{\ord_p(q_p^{\Tate})}.$$
Formula \eqref{formula} was proven by Greenberg and Stevens (cf.~\cite{GS}) utilizing the two-variable Kitagawa-Mazur $p$-adic $L$-function, which is defined by varying the modular form associated with $A$ in a Hida family.

In \cite{BD99} Bertolini and Darmon prove an analogue of \eqref{formula} for the anticyclotomic $p$-adic $L$-function of the base change of $A$ to an imaginary quadratic field in which $p$ splits.
In contrast to the proof of Greenberg and Stevens and many other proofs of similar exceptional zero formulae, they do not make use of $p$-adic families of automorphic forms.
Instead they use the Cerednik-Drinfeld uniformization of Shimura curves.
The main purpose of this note is to generalize their result to Hilbert modular forms of parallel weight $2$.
Let us give a more detailed account of the content of this article:

We fix a quadratic extension $E/F$ of number fields with totally real base field $F$.
In addition, we fix a non-split quaternion algebra $B/F$ in which $E$ can be embedded.
We assume that at the Archimedean places of $F$ the algebra $B$ is split if an only if $E$ is split.
Let $\pi_B$ be an automorphic representation of $B^{\ast}/F^{\ast}$ that is cohomological with respect to the trivial coefficient system.
For every allowable modulus $\m$ (see Definition \ref{allowable}) we can pull back the cohomology class $\kappa$ associated to the new-vector of $\pi_B$ via an embedding $T=E^{\ast}/F^{\ast}\too B^{\ast}/F^{\ast}$ of conductor $\m$ to get a distribution valued cohomology class $\Delta_{\m}(\kappa)$.

Let $L/E$ be an anticyclotomic extension with Galois group $\G$.
Following Das\-gup\-ta and Spie\ss~the cap product of the Artin reciprocity map with a fundamental class for the group of relative units of $E/F$ gives a homology class $c_L$.
Let $R_\pi$ be the ring of integers of the field of definition of $\pi_B$. By the assumption on the splitting behaviour at infinite places we can define the Stickelberger element $$\Theta_{\mathfrak{m}}(L/F,\kappa)=\Delta_{\mathfrak{m}}(\kappa)\cap c_L\in R_\pi[\G].$$
An analysis of the action of local points of the torus $T$ on Bruhat-Tits trees gives functional equations for Stickelberger elements (Proposition \ref{globfunceq}) and norm relations between Stickelberger elements of different moduli (Theorem \ref{compatibility}).
Some crucial calculations on the tree have already been carried out by Cornut and Vatsal in \cite{CV}.

In Section \ref{Interpolation} we use results of File, Martin and Pitale (cf.~\cite{FMP}) on toric period integrals to show that our Stickelberger elements interpolate (square roots of) special values of the $L$-function of the base change of $\pi_B$ with respect to $E/F$, i.e.~for every character $\chi\colon \G\to\C^{\ast}$ of conductor $\m$ we have $$|\chi(\Theta_{\mathfrak{m}}(L/F,\kappa))|^{2}\stackrel{\cdot}{=} L(1/2,\pi_{B,E}\otimes \chi),$$
where ``$\stackrel{\cdot}{=}$'' means equality up to explicit fudge factors.
In the CM case, i.e.~if the field $E$ is totally imaginary, our construction is closely related to van Order's construction of $p$-adic $L$-functions (see \cite{vO}).
Note that, if $E$ is totally real, one expects that there exists no anticyclotomic $\Z_\p$-extension but anticyclotomic Stickelberger elements at finite level can still be defined.

For a ring $R$ and and abelian group $H$ let $I_{R}(H)\subseteq R[H]$ be the kernel of the augmentation map $R[H]\to R$ given by $h\mapsto 1$.
We define the order of vanishing $\ord_R(\xi)$ of an element $\xi \in R[H]$ as
\begin{align*}
 \ord_R(\xi) = \begin{cases}
                r 	& \text{ if } \xi \in I_{R}(H)^r - I_{R}(H)^{r+1}, \\
                \infty 	& \text{ if } \xi \in I_{R}^{j}(H)\ \mbox{for all}\ j\geq 1.
               \end{cases}              
\end{align*}
By generalizing methods developed by Dasgupta and Spie\ss~ in \cite{DS} from the split torus to an arbitrary one-dimensional torus we show that $$\ord_{R_\pi}(\Theta_{\mathfrak{m}}(L/F,\kappa))\geq |S_\m|,$$ where $S_\m$ is the set of all primes $\p$ of $F$ such that $\p\mid \m$ and either $\pi_{B,\p}$ is Steinberg or $\p$ is inert in $E$ and $\pi_{B,\p}$ is the non-trivial unramified twist of the Steinberg representation.
More generally, in Theorem \ref{ordvanish} we show that the above anticyclotomic Stickelberger element lies in a product of partial augmentation ideals. 

In Section \ref{leadingterms} we prove the following leading term formula: Suppose that all primes $\p$ in $S_\m$ are split in $E$ and $\pi_{B,\p}$ is Steinberg.
Then we define ``automorphic periods'' $\q_\p\in F_\p^{\ast}\otimes R_\pi$ such that
$$\prod_{\p\in S_{\m}} \ord_\p(q_\p)\cdot\Theta_{\mathfrak{m}}(L/F,\kappa)
\stackrel{\cdot}{=}\prod_{\p\in S_{\m}}(\rec_{\p}(q_\p)-1)\cdot \sqrt{L(1/2,\pi_{B,E})}\bmod I_{R_\pi}(\G)^{|S_\m|+1}$$
holds, where $\rec_\p$ is the local reciprocity map at a prime of $E$ lying above $\p$.
This generalizes Molina's work \cite{Santi} on exceptional zeros of anticyclotomic $p$-adic $L$-functions in the CM case and is heavily inspired by Spie\ss' article \cite{Sp}.
Crucial in the definition of these automorphic periods are extension classes of the Steinberg representation, which were first studied by Breuil in \cite{Br}.
More precisely, for each continuous homomorphism $l\colon F_{\p}^{\ast}\to A$ to some topological group $A$ we construct a short exact sequence $$0\too \St_\p(A)\too \mathcal{E}(l)\too\Z\too 0,$$ where $\St_\p(A)=C(\PP(F_\p),A)/A$ is the $A$-valued continuous Steinberg representation.

Finally, if $E$ is totally imaginary and therefore, $B$ is totally definite, we compare automorphic periods of $\pi_B$ and Tate periods of the abelian variety associated to $\pi_B$.
The first step is to show that we have a commutative diagram
\begin{center}
 \begin{tikzpicture}
    \path 	(0.5,0) node[name=A]{$0$}
		(3,0) node[name=B]{$\Div^0(\mathcal{H}_\p(\C_p))$}
		(6,0) node[name=C]{$\Div(\mathcal{H}_\p(\C_p))$}
		(8.5,0) node[name=D]{$\Z$}
		(10.5,0) node[name=E]{$0$}
		(0.5,-1.5) node[name=F]{$0$}
		(3,-1.5) node[name=G]{$\St_\p(\C_p^{\ast})$}
		(6,-1.5) node[name=H]{$\mathcal{E}(\sigma_\p)$}
		(8.5,-1.5) node[name=I]{$\Z$}
		(10.5,-1.5) node[name=J]{$0$};
    \draw[->] (A) -- (B) ;
		\draw[->] (B) -- (C) ;
		\draw[->] (C) -- (D) ;
		\draw[->] (D) -- (E) ;
		\draw[->] (F) -- (G) ;
		\draw[->] (G) -- (H) ;
		\draw[->] (H) -- (I) ;
		\draw[->] (I) -- (J) ;
    \draw[->] (B) -- (G) node[midway, right]{$\Psi$};
		\draw[->] (C) -- (H);
    \draw[->] (D) -- (I) node[midway, right]{$=$};
  \end{tikzpicture}
\end{center}
with exact rows.
Here $\mathcal{H}_\p$ denotes the $p$-adic upper half plane, $\Psi$ is the map that associates to each divisor $D$ a rational function with divisor $D$ and $\sigma_\p\colon F_{\p}^{\ast}\to \C_{\p}^{\ast}$ is a fixed embedding.
Combined with Dasgupta's variant of the Manin-Drinfeld uniformization theorem, which was implicitly already given in \cite{BD98}, we get a description of the Jacobian of certain Mumford curves via the universal extension of the Steinberg representation (see Theorem \ref{unif}).
By applying this description to Jacobians of Shimura curves associated to quaternion algebras $\bar{B}$, whose invariants only differ from that of $B$ at $\p$ and one Archimedean place, we are able to compare the different periods.

\bigskip
\textbf{Acknowledgments}
It is our pleasure to thank Chan Ho-Kim for giving an excellent talk at the workshop on 'Arithmetic of Euler systems' in Benasque, which encouraged us to start this project.
We are grateful to Jeanine van Order for providing us with the reference \cite{FMP}.
While writing this note Santiago Molina informed us that he proved a similar result on the equality of automorphic and algebraic periods.
We would like to thank him for providing us with a preliminary manuscript of his results.
We thank the anonymous referee, whose detailed comments helped to improve the exposition of our results. 
The first named author was financially supported by the DFG within the CRC 701 'Spectral Structures and Topological Methods in Mathematics'.

\bigskip
\textbf{Notations}.
We will use the following notations throughout the article.
All rings are commutative and unital.
The group of invertible elements of a ring $R$ will be denoted by $R^{\ast}$.
Given a group $H$ and a group homomorphism $\chi:H\to R^{\ast}$ we let $R(\chi)$ be the representation of $H$ whose underlying $R$-module is $R$ itself and on which $H$ acts via the character $\chi$.
If $N$ is another $R[H]$-module, we put $N(\chi)=N\otimes_{R}R(\chi)$.
Let $\Theta$ be an element of $R[H]$.
We write $\Theta^{\vee}$ for the image of $\Theta$ under the map induced by inversion on $H$.

For a set $X$ and a subset $A\subseteq X$ the characteristic function $\cf_{A}\colon X\to \left\{0,1\right\}$ is defined by
\begin{align*}
 \cf_{A}(x) = \begin{cases}
	       1 & \mbox{if } x\in A,\\
	       0 & \mbox{else}.
	      \end{cases}
\end{align*}

Given topological spaces $X,Y$ we will write $C(X,Y)$ for the set of continuous functions from $X$ to $Y$.
For a topological ring $R$ we define $C_{c}(X,R) \subseteq C(X,R)$ as the subset of continuous functions with compact support.
If we consider $Y$ (resp.~$R$) with the discrete topology, we will often write $C^{0}(X,Y)$ (resp.~$C_{c}^{0}(X,R)$) instead.

For a ring $R$ and an $R$-module $N$, we define the $R$-module of $N$-valued distributions on $X$ as $\Dist(X,N)=\Hom_{\Z}(C^{0}_{c}(X,\Z),N)$.
If $X$ is discrete, we have the pairing
\begin{align*}
C_{c}^{0}(X,\Z)\times C^{0}(X,N)\too N \intertext{ given by } (\psi,\phi)\mapstoo \sum_{x\in X} (\psi\cdot\phi)(x),
\end{align*}
which induces an isomorphism of $R$-modules
\begin{align}\label{discrete}
C^{0}(X,N)\xlongrightarrow{\cong}\Dist(X,N).
\end{align}
We will always identify these two $R$-modules via the above isomorphism if $X$ is discrete.
In the case that $X$ is a compact space, we denote the space of $N$-valued distributions of total volume $0$ by $\Dist_0(X,N)$.

We say that an $R$-module $N$ is prodiscrete if $N$ is a topological group such that there exist open $R$-submodules $$\ldots\subseteq N_2\subseteq N_1\subseteq N$$ with $\bigcap_i N_i =\left\{0\right\}$ and $N=\varprojlim_{i}N/N_i$.
Let $X$  be a totally disconnected compact space $X$ and $N$ a prodiscrete $R$-module.
We restrict the canonical pairing $$\varprojlim_{i}C^{0}(X,N/N_i)\otimes \Dist(X,R)\too \varprojlim_{i} N/N_i=N$$ to $C(X,N)$ via the embedding
$$C(X,N)\hooklongrightarrow \varprojlim_{i}C^{0}(X,N/N_i).$$
This yields an integration pairing
\begin{align}\label{integration}
 C(X,N)\otimes \Dist(X,R)\too N.
\end{align}

Throughout the article we fix a totally real number field $F$ with ring of integers $\OO_{F}$. 
For a non-zero ideal $\mathfrak{a} \subseteq \OO_F$ we set $N(\mathfrak{a})=\left|\OO_F/\mathfrak{a}\right|$.
If $v$ is a place of $F$, we denote by $F_{v}$ the completion of $F$ at $v$.
If $\p$ is a finite place, we let $\OO_{F_\p}$ denote the valuation ring of $F_\p$ and write $\ord_{\p}$ for the normalized valuation.

For a finite (possibly empty) set $S$ of places of $F$ we define the "$S$-truncated adeles" $\A^{S}$ as the restricted product of the completions $F_{v}$ over all places $v$ which are not in $S$. We often write $\A\sinfty$ instead of $\A^{S\cup S_{\infty}}$.
Here $S_{\infty}$ denotes the set of Archimedean places of $F$. If $G$ is an algebraic group over $F$ and $v$ is a place of $F$, we write $G_v=G(F_v)$ and put $G_S=\prod_{v\in S} G_v$. 
Furthermore, if $K\subseteq G(\A)$ is a subgroup, we define $K^{S}$ as the image of $K$ under the quotient map $G(\A)\to G(\A^{S})$.
If $S$ is a set of finite places of $F$ and $\m\subseteq \OO_F$ is a non-zero ideal we put
$$S_\m=\left\{\p\in S\mid \p\ \mbox{divides}\ \m\right\}.$$

\section{Anticyclotomic characters and homology classes}\label{first}
In Section 3 of \cite{DS} Dasgupta and Spie\ss~develop a machinery to bound the order of vanishing of Stickelberger elements coming from distributions on the split one-dimensional torus.
In this section we indicate how to generalize their methods to non-split tori.
At primes at which the torus splits essentially the same arguments as in \cite{DS} apply.
At a non-split prime $\p$ the situation turns out to be even simpler: the local torus is compact and thus, the rank does not change if one passes from arithmetic subgroups to $\p$-arithmetic subgroups of the torus.

Let us fix a quadratic extension $E$ of $F$ with ring of integers $\OO$ and Galois group generated by $\tau$.
We write $d$ for the number of Archimedean places of $F$ which are split in $E$.
For a finite place $\p$ of $F$ all $\OO_{F_\p}$-orders in $E_\p$ are of the form $\OO_{F_\p} + \p^m \OO_{\p}$ for some $m\geq 0$, where $\OO_{\p}$ denotes the maximal $\OO_{F_\p}$-order in $E_\p$.

We consider the algebraic torus $T=E^\ast/F^\ast$ over $F$.
If $\p$ is a finite place of $F$, we write $U_{T_\p}^{(m)}$ for the image of $(\OO_{F_\p} + \p^m \OO_{\p})^{\ast}$ in $T_{\p}$.
If $v$ is an Archimedean place of $F$, we define $U_{T_v}$ as the connected component of $1$ in $T_{v}$.
Further, we put $$U_{T_{\infty}}=\prod_{v \in S_{\infty}} U_{T_v} \subset T_\infty.$$
Given a non-zero ideal $\m \subseteq \OO_F$ we define $$U_T(\m) = \prod_{\p \notin S_{\infty}} U_{T_\p}^{(\ord_\p(\m))} \times U_{T_{\infty}}\subseteq T(\A).$$
To ease the notation we write $U_T$ instead of $U_T(\OO_F)$.

For every prime $\p$ of $F$ we fix once and for all a prime $\P$ of $E$ lying above $\p$ and a local uniformizer $\varpi_{\P}$ at $\P$.
If $\p$ is split, the choice of $\P$ determines an isomorphism $T_\p\cong F_{\p}^{\ast}$. We will always identify these two groups via the above isomorphism.
Likewise, for every split Archimedean place $v$ of $F$ we fix a place $w$ of $E$ above $v$ and identify $T_v$ with $F_v^{\ast}$.

\subsection{Fundamental classes} \label{Fundamental}
Suppose that there exists an Archimedean place $v$ of $F$ which splits in $E$.
The group $U_{T_v}\cong \R^{\ast}_{>0}$ is torsion-free.
Therefore, for every subgroup $A \subseteq T(F)$ the group
$$A^{+}=\ker\left(A\too T_{\infty}/ U_{T_\infty}\right)$$
is torsion-free.

In the CM case, i.e.~there is no Archimedean place that splits in $E$, we choose an auxiliary finite place $\q$ of $F$ and a maximal open torsion-free subgroup $U_{T_\q}^+\subseteq U_{T_\q}$.
If $A\subseteq T(F)$ is a subgroup such that the image of $A$ under the embedding $T(F)\into T_{\q}$ is contained in $U_{T_{\q}}$, we define
$$A^{+}=\ker\left(A\too U_{T_{\q}}/ U_{T_\q}^+\right).$$
Similarly, if $\widetilde{U} \subseteq U_{T}$ is any subgroup, we define $\widetilde{U}^+\subseteq \widetilde{U}$ to be the subgroup of elements which $\q$-component lies in $U_{T_\q}^+$.
To avoid distinguishing the two cases we simply put $\widetilde{U}^+=\widetilde{U}$ if there is one Archimedean place that splits.

\begin{Remark}
For the rest of the article we use the following convention in the CM case: whenever we choose a set of finite primes $S$ (resp.~a non-zero ideal $\m$) of $F$ we will assume that the fixed prime $\q$ is not contained in $S$ (resp. co-prime to $\m$).
\end{Remark}

Given a finite (possibly empty) set $S$ of places of $F$, an open subgroup $\widetilde{U} \subseteq U_T^{S}$ and a ring $R$ we define
\begin{align*}
 \mathcal{C}_{?}(\widetilde{U},R)^S=C_{?}^{0}(T(\A^{S})/\widetilde{U}^+,R)
\end{align*}
for $?\in\left\{\emptyset, c\right\}$.
For a non-zero ideal $\m\subseteq \OO_F$ we set
\begin{align}\label{funcspac}
 \mathcal{C}_{?}(\m,R)^S=\mathcal{C}_{?}(U_T(\m)^{S},R)^S.
\end{align}
If $S$ is the empty set, we drop it from the notation.

Further, we define $$\mathcal{U}_S=\Ker\left(T(F)\too T(\A^{S})/U^{S}_{T}\right).$$
By Dirichlet's unit theorem $\mathcal{U}_{S}^+$ is a free group of rank $d+r$, where $r$ is the number of places in $S$ which are split in $E$.
Thus, the homology group $\HH_{d+r}(\mathcal{U}_{S}^+,\Z)$ is free of rank one. We fix a generator $\eta^{S}$ of this group.
Further, we fix a fundamental domain $\mathcal{F}^S$ for the action of $T(F)/\mathcal{U}_{S}^+$ on $T(\A^{S})/U^{S,+}_{T}$.
By Shapiro's lemma the identification $$\mathcal{C}_c(\OO,\Z)^{S}=\cind_{\mathcal{U}_{S}^+}^{T(F)}C(\mathcal{F}^S,\Z)$$ induces an isomorphism $$\HH_{d+r}(\mathcal{U}_{S}^{+},C(\mathcal{F}^S,\Z))\xlongrightarrow{\cong} \HH_{d+r}(T(F),\mathcal{C}_c(\OO,\Z)^{S}).$$
The fundamental class $\vartheta^{S}$ is defined as the image of the cap product of $\eta^{S}$ with the characteristic function $\cf_{\mathcal{F}^S}$ under the above isomorphism.
Similarly as before, we drop the superscript $S$ if it is the empty set.

\begin{Remark}\label{fundamentaltori}
Let $S^{+}\subseteq S_\infty$ be the set of all split Archimedean places of $F$.
A generator $\eta$ of $\HH_{d}(\mathcal{U}^+,\Z)$ can be identified with the fundamental class of the compact torus $U_{T_{S^{+}}} /\mathcal{U}^+$.
\end{Remark}

If $\p\in S$ is either inert or ramified in $E$, the group $\HH^{0}(T_\p,C(T_{\p}/U_{T_{\p}},\Z))$ is free of rank one.
Let $c_{\p}$ be the normalized generator, i.e.~ the function that is constantly one.
If $\p\in S$ is split, we have a sequence of $T_\p$-modules
\begin{align}\label{fundamentalexact}
 0\too C_c(F_{\p}^{\ast},\Z)\too C_c(F_\p,\Z)\xlongrightarrow{g\mapsto g(0)}\Z\too 0,
\end{align}
where the torus acts via the identification $T_\p\cong F_{\p}^{\ast}$.
Taking $U_{T_\p}$-invariants yields the exact sequence
\begin{align}\label{fundexactinv}
 0\too C_c(F_{\p}^{\ast}/U_{T_\p},\Z)\too C_c(F_\p,\Z)^{U_{T_\p}}\xlongrightarrow{g\mapsto g(0)}\Z\too 0.
\end{align}
We define $c_{\p}$ as the image of $1\in \Z$ under the connecting homomorphism $\Z\to \HH^{1}(T_\p,C_c(T_{\p}/U_{T_{\p}},\Z))$.

\begin{Remark}\label{kindofproof}
If $\p$ is split, the group $T_{\p}/U_{T_\p}$ is a free abelian group of rank $1$.
The exact sequence \eqref{fundexactinv} is a projective resolution of the trivial $T_\p/U_{T_\p}$-module.
Therefore, if $\eta_\p$ is a generator of the free abelian group $\HH_{1}(T_\p/U_{T_\p},\Z)$ of rank $1$ we get
$$c_\p\cap \eta_\p=\pm1 \in \HH_{0}(T_\p/U_{T_\p},C_c(T_{\p}/U_{T_{\p}},\Z))\cong \Z.$$
\end{Remark}

The canonical pairing $$\mathcal{C}_c(\OO,\Z)^{S}\times C_c(T_\p/U_{T_\p},\Z)\too \mathcal{C}_c(\OO,\Z)^{S - \left\{\p\right\}}$$
induces a cap product pairing on (co)homology groups.
The following lemma essentially follows from Remark \ref{kindofproof}. 
\begin{Lemma}\label{fundclasses}
For every $\p\in S$ the equality $\vartheta^{S - \left\{\p\right\}}=\pm c_\p \cap \vartheta^{S}$ holds.
The sign only depends on the choice of the generators $\eta^{S}$ and $\eta^{S - \left\{\p\right\}}$.
\end{Lemma}

\subsection{Derivatives of local characters} \label{LocalChar}
In this section we fix a finite place $\p$ of $F$.
Let $A$ be a group and $l_\p\colon T_\p\to A$ a locally constant homomorphism.
We can view $l_\p$ as an element of $C^{0}(T_\p,A)$.
Since $l_\p$ is a group homomorphism the map $y\mapsto y.l_\p-l_\p$ is constant.
Thus, the image of $l_\p$ in $C^{0}(T_\p,A)/A$ is fixed by the $T_\p$-action.

If $\p$ is inert or ramified in $E$, we define 
\begin{align*}
 c_{l_{\p}}\in \HH^{0}(T_\p,C^{0}(T_\p,A)/A)
\end{align*}
to be the image of $l_\p$.

On the other hand, if $\p$ is split in $E$, we define
\begin{align*}
 c_{l_{\p}}\in \HH^{1}(T_\p,C_c^{0}(F_\p,A))
\end{align*}
to be the class given by the cocycle  
\begin{align} \label{zchip}
 z_{l_{\p}}(x)(y)=\cf_{x\mathcal{O}_{F,\p}}(y)\cdot l_{\p}(x) +((\cf_{\mathcal{O}_{F,\p}}-\cf_{x\mathcal{O}_{F,\p}})\cdot l_\p)(y)
\end{align}
for $x\in T_{\p}$ and $y\in F_{\p}$.

\begin{Remark}\label{loccomp}
For a prime $\p$, which is split in $E$, we consider the unique $T_\p$-equivariant homomorphism
\begin{align}\label{loccompmap}
\alpha_{\p}\colon C_c(T_\p/U_{T_\p},\Z)\too C_c(F_\p,\Z)
\end{align}
that sends $\cf_{U_\p}$ to $\cf_{\OO_{F_\p}}$.
The class $c_{\ord_{\p}}$ associated to the homomorphism $$\ord_{\p}\colon T_{\p}\cong F_{\p}^{\ast}\too \Z$$ is equal to the image of the class $c_\p$ under the homomorphism
$$\HH^{1}(T_{\p},C_c(T_\p/U_{T_\p},\Z))\xlongrightarrow{(\alpha_{\p})_{\ast}} \HH^{1}(T_{\p},C_c(F_\p,\Z)).$$
More generally, the class $c_{l_{\p}}$ can be constructed as the image of $l_{\p}$ under a boundary map $$\delta\colon \HH^{0}(T_\p,C^{0}(T_\p,A)/A)\too \HH^{1}(T_\p,C_c^{0}(F_\p,A)).$$
See Section 3.2 of \cite{DS} for more details.
\end{Remark}
We are mostly interested in the following situation.
We fix a ring $R$ and an ideal $\mathfrak{a} \subseteq R$.
We set $\overline{R}=R/\mathfrak{a}$ and similarly, we write $\overline{N}=N\otimes_{R}\overline{R}$ for every $R$-module $N$. 
Let $\chi_\p\colon T_\p\to R^{\ast}$ be a character.
Suppose we have given an ideal $\mathfrak{a}_\p \subseteq \mathfrak{a}$ such that $\chi_\p \equiv 1 \bmod \mathfrak{a}_\p$.
Then
\begin{align*}
 \dd\chi_\p\colon T_\p\too \overline{\mathfrak{a}_\p},\ x\mapstoo \chi_\p(x)-1 \bmod \mathfrak{a}\mathfrak{a}_\p.
\end{align*}
defines a group homomorphism, which yields a cohomology class $c_{\dd\chi_\p}$.

\subsection{Derivatives of global characters} \label{Vanishing}
As above we fix a ring $R$ and an ideal $\mathfrak{a} \subseteq R$.
Let $\chi\colon T(\A)/T(F)\to R^{\ast}$ be a locally constant character and
write $\overline{\chi}\colon T(\A)/T(F) \to \overline{R}^{\ast}$ for its reduction modulo $\mathfrak{a}$.
For a place $v$ of $F$ we denote by $\chi_{v}$ the local component of $\chi$ at $v$, i.e.~the composition $$\chi_{v}\colon T_v \hooklongrightarrow T(\A)  \xlongrightarrow{\chi} R^{\ast}.$$
Since the kernel of $\chi$ is open there exists a non-zero ideal $\mathfrak{m} \subseteq \mathcal{O}$ such that $\chi$ restricted to $U_T(\mathfrak{m})$ is trivial. The smallest such ideal is called the conductor of $\chi$. Similarly, for a finite place $\p$ of $F$ we define the conductor of $\chi_\p$ to be the $\p$-component of the conductor of $\chi$.
We will fix such an ideal $\m$ (not necessarily the conductor) in the following and view $\chi$ as an element of $\HH^{0}(T(F),\mathcal{C}(\mathfrak{m},R))$.

Suppose we have given a finite set $S$ of finite places of $F$ and ideals $\mathfrak{a}_\p \subseteq \mathfrak{a}$ for $\p\in S$ such that $\chi_\p\equiv 1 \bmod \mathfrak{a}_\p$ holds.
In this case, we can regard the restriction $\overline{\chi}^{S}$ of $\overline{\chi}$ to $T(\A^{S})$ as an element of $\HH^{0}(T(F),\mathcal{C}(\m,\overline{R})^{S})$.
Further, we want to take the Archimedean places into account.
Let $\overline{\chi}\sinfty$ be the restriction of $\overline{\chi}$ to $T(\A\sinfty)$
For every Archimedean place $v$ of $F$ which is split in $E$ we fix a character $\epsilon_v\colon T(F_v)/U_{T_v}\to \left\{\pm 1\right\}$ and an ideal $\mathfrak{a}_v \subseteq \mathfrak{a}$ with
$$\chi_v(-1)\equiv -\epsilon_v(-1) \bmod \mathfrak{a}_v.$$
Thus, $\psi_v:=1+(\chi_v\epsilon_v)(-1)$ is an element of $\mathfrak{a}_v$. 
If $v$ is non-split, we set $\psi_v=\epsilon_v=1$, and $\mathfrak{a}_v=R$.
Let us write $\epsilon=\prod_{v\in S_\infty}\epsilon_v\colon T_\infty\to \left\{\pm 1\right\}$.
An easy calculation shows that $\widetilde{\chi}^{S}:=\prod_{v\in S_\infty} \psi_v \cdot \overline{\chi}\sinfty$ defines an element of $\HH^{0}(T(F),\mathcal{C}(\m,\overline{\prod_{v\in S_\infty}\mathfrak{a}_v})\sinfty(\epsilon))$.

The $\epsilon$-isotypical projection
\begin{align*}
 C(T_{\infty}/U_{T_\infty},R)\too R(\epsilon),\ f\mapstoo \sum_{x\in T_\infty/U_{T_{\infty}}} \epsilon(x)f(x)
\end{align*}
yields a $T(\A)$-equivariant map
\begin{align}\label{sign}
 \mathcal{C}_{c}(\mathfrak{m},R)\too \mathcal{C}_{c}(\mathfrak{m},R)^{\infty}(\epsilon).
\end{align}

Let $r$ be the number of primes in $S$ which are split in $E$.
We define $\vartheta\sinfty$ to be the image of $\vartheta^{S}$ under the map
\begin{align*}
 \HH_{d+r}(T(F),\mathcal{C}_c(\OO,\Z)^{S})\too \HH_{d+r}(T(F),\mathcal{C}_c(\OO,\Z)\sinfty)
\end{align*}
induced by \eqref{sign} with $\epsilon=1$.
Furthermore, we define $$\overline{c_\chi}=\overline{c_{\chi}}(\mathfrak{m},S,\epsilon)\in \HH_{d+r}(T(F),\mathcal{C}_c(\m,\overline{\prod_{v\in S_\infty}\mathfrak{a}_v})\sinfty(\epsilon))$$ as the cap product of $\widetilde{\chi}^{S}$ with $\vartheta\sinfty$.

Next, we are going to attach a homology class $c_\chi$ to the character $\chi$ and compare it with the class $\overline{c_\chi}$ associated to its reduction.
For this, we need to consider a slight generalization of \eqref{funcspac}.
Let $S'$ be another (possibly empty) finite set of finite places of $F$ disjoint from $S$.
For an open subgroup $\widetilde{U} \subseteq U_T^{S\cup S^{\prime},\infty}$ we define
\begin{align*}
 \mathcal{C}_{c}(\widetilde{U},S,R)^{S',\infty} = \mathcal{C}_{c}(\widetilde{U},R)^{S \cup S', \infty} \otimes \hspace{-1em}\bigotimes_{\substack{\p \in S \\ \p\ \text{non-split}}}\hspace{-1em}C^{0}_{c}(T_{\p},R)/R\otimes \bigotimes_{\substack{\p \in S\\ \p\ \text{split}}}C^{0}_c\left(F_{\p},R\right).
\end{align*}
As before, we put $\mathcal{C}_{c}(\m,S,R)^{S',\infty}=\mathcal{C}_{c}(U_T(\m)^{S\cup S^{\prime},\infty},S,R)^{S',\infty}$ and drop $S'$ from the notation if it is the empty set.
Extension by zero at the split places together with the canonical projection at non-split places induces a map
\begin{align}\label{extension}
 \mathcal{C}_{c}(\mathfrak{m},R)^{\infty}\too \mathcal{C}_{c}(\mathfrak{m},S,R)^{\infty}.
\end{align}
Let $c_\chi=c_{\chi}(\mathfrak{m},S,\epsilon)$ denote the image of $\chi$ under the composition
\begin{align} \label{homclass}
 \begin{split}
 H^{0}(T(F),\mathcal{C}(\mathfrak{m},R)) &\xlongrightarrow{\cdot\cap \vartheta\ } \HH_{d}(T(F),\mathcal{C}_{c}(\mathfrak{m},R))\\
					 &\xlongrightarrow{\eqref{sign}} \HH_{d}(T(F), \mathcal{C}_{c}(\mathfrak{m},R)^{\infty}(\epsilon))\\
					 &\xlongrightarrow{\eqref{extension}} \HH_{d}(T(F), \mathcal{C}_{c}(\mathfrak{m},S,R)^{\infty}(\epsilon)).
 \end{split}
\end{align}

From now on we assume that $\mathfrak{a}\cdot \prod_{v\in S\cup S_\infty} \mathfrak{a}_v= 0$.
Hence, multiplication in $R$ induces a multilinear map
\begin{align*}
\mu\colon \overline{\mathfrak{a}_{\p_1}}\times\ldots\times \overline{\mathfrak{a}_{\p_s}}\times \overline{\prod_{v\in S_\infty}\mathfrak{a}_v} \too \prod_{v\in S\cup S_{\infty}}\mathfrak{a}_v\hooklongrightarrow R,
\end{align*}
where $S=\left\{\p_1,\ldots, \p_s\right\}$.
The next proposition can be proved along the same lines as Proposition 3.8 of \cite{DS}.

\begin{Proposition} \label{vanish1}
The following equality of homology classes holds: 
\begin{align*}
 c_\chi=\pm\ \mu_{\ast}((c_{\dd\chi_{\p_1}}\cup\ldots \cup c_{\dd\chi_{\p_{s}}})\cap \overline{c_\chi})
\end{align*}
In particular, $c_\chi=0$ if $\prod_{v\in S\cup S_{\infty}}\mathfrak{a}_v=0.$
\end{Proposition}

\section{Quaternionic Stickelberger elements}
In this section we study Stickelberger elements coming from cohomology classes of arithmetic subgroups of the multiplicative group of quaternion algebras.
Stickelberger elements are constructed by firstly, pulling back these classes via embeddings of the multiplicative group of $E$ into the one of the quaternion algebra and secondly, taking cap products with homology classes associated to the Artin reciprocity map.
As an immediate consequence of the results of Section \ref{first} we can bound their order of vanishing from below.

Let us fix a non-split quaternion algebra $B$ over $F$ such that
\begin{itemize}
 \item $E$ can be embedded into $B$, i.e.~all places of $F$ at which $B$ is non-split are non-split in $E$ as well and
 \item $B$ is non-split at all Archimedean places of $F$ which are non-split in $E$.
\end{itemize}
The set of finite places of $F$ at which $B$ is ramified will be denoted by $\ram(B)$.
We choose once and for all an embedding $\iota \colon E \hookrightarrow B$.
By the Skolem-Noether-Theorem there exists a $J \in B^{\ast}$, unique up to multiplication by an element of $E^{\ast}$, such that $$\iota(\tau(e))= J \iota(e) J^{-1}$$ holds for all $e\in E$.
Let us fix such an element $J \in B^{\ast}$.
We consider the reductive $F$-algebraic group $G=B^\ast/F^\ast$ and view $T$ as an algebraic subgroup of $G$ via the embedding $\iota$. 
In addition, we fix a maximal order $\mathcal{R}\subseteq B$ such that $\iota(\OO)\subseteq \mathcal{R}$.
For all primes $\p$ of $F$ we write $\mathcal{R}_\p\subseteq B_\p$ for the induced maximal order and $K_\p$ for the image of $\mathcal{R}_\p^{\ast}$ in $G_\p$.
Let $\p\in\ram(B)$ be a prime which is inert in $E$.
From the explicit description of the non-split quaternion algebra over a $p$-adic local field one gets that the element $J$ is a $E_\p^\ast$-multiple of a uniformizer of a ramified quadratic extension of $F_\p$.
Therefore, we have $J\notin K_\p$ in this case. 
\subsection{Local norm relations} \label{Local}
This section contains all local computations that we need to prove norm relations between Stickelberger elements of different moduli and functional equations for Stickelberger elements.
If the prime under consideration is split in $E$, explicit versions of the following computations are given in \cite{BG}.
Most local norm relations were already proven by Cornut and Vatsal in Section 6 of \cite{CV}.
 
We fix a finite place $\p$ of $F$ at which $B$ is split.
In particular, the group $G_\p$ is isomorphic to $\PGL_2(F_\p)$.
Let $\T_\p=(\V_\p,\E_\p)$ be the Bruhat-Tits tree of $G_\p$, i.e.~
\begin{itemize}
 \item $\V_\p$ is the set of maximal orders in $B_\p$ and
 \item there exists an oriented edge $e=(v,v')\in\E_\p$ between two vertices $v,v' \in \V_\p$ if and only if the intersection of the corresponding orders is an Eichler order of level $\p$.
\end{itemize}
Note that $(v,v') \in \E_\p$ if and only if $(v',v) \in \E_\p$.
In this situation we say that $v$ and $v'$ are neighbours and write $v \sim v'$.
Each vertex has $N(\p)+1$ neighbours.

For an integer $n\geq 0$ we define $\E_{\p,n}$ as the set of non-backtracking paths in $\T_\p$ of length $n$, i.e.~
\begin{align*}
 \E_{\p,n} = \left\{ (v_0,\dots,v_n) \in \V_\p^{n+1} \mid (v_i,v_{i+1}) \in \E_\p\ \mbox{and}\ v_i \neq v_{i+2}\ \mbox{for all}\ i \right\}.
\end{align*}
In particular, we have $\E_{\p,0}=\V_\p$ and $\E_{\p,1}=\E_\p$. The group $G_\p$ acts on $\E_{\p,n}$ via conjugation in each component.

Let $R$ be a ring and $N$ an $R$-module.
In the following we consider $\E_{\p,n}$ as a discrete topological space.
The Atkin-Lehner involution $W_{\p^{n}}$ on $C^{0}(\E_{\p,n},N)$ is given by interchanging the orientation, i.e. $$W_{\p^{n}}(\phi)(v_0,\dots,v_n) = \phi(v_n,\dots,v_0).$$
The Hecke operator 
\begin{align}\label{Heckeoperatoren}
 \mathbb{T}_\p \colon C^{0}(\E_{\p,n},N)	&\too C^{0}(\E_{\p,n},N)
\end{align}
is defined by
\begin{align*}
 \phi 				&\mapstoo \left((v_0,\dots,v_n) \mapstoo 
					      \sum\limits_{v_{n-1} \neq v \sim v_{n}} \phi(v_1,\dots,v_{n},v)\right).
\end{align*}
Note that, if $n=0$, the condition $v_{n-1}\neq v$ is empty.
For $(v_0,\dots,v_n) \in \E_{\p,n}$ we define
\begin{align*}
 \partial_{(v_0,\dots,v_n)} \colon C^{0}(\E_{\p,n},N) 	&\too\Dist(T_\p/\Stab_{T_\p}((v_0,\dots,v_n)),N)
\intertext{to be the $T_\p$-equivariant map given by}
 \phi					&\mapstoo \left(t \mapsto \phi(t.(v_0,\dots,v_n)) \right).
\end{align*}
Here we used the identification \eqref{discrete} of distribution and function spaces on the discrete space $T_\p/\Stab_{T_\p}((v_0,\dots,v_n))$.
For $v \in \V_\p$ let $l_\p(v)$ be the uniquely determined integer given by $$U_{T_\p}^{(l_\p(v))} = \Stab_{T_\p}(v).$$

\begin{Remark} \label{EichlerTree}
Let $\mathcal{R}_\p(\p^{n}) \subseteq B_\p$ be an Eichler order of level $\p^n$ contained in the fixed maximal order $\mathcal{R}_\p$.
We write $K_{\p}(\p^{n})$ for the image of $\mathcal{R}_\p(\p^{n})^{\ast}$ in $G_\p$.
There exists a unique vertex in $\V_\p$ fixed by $K_{\p}$ and thus, we get a canonical isomorphism $$C(G_\p/K_\p,N)\xlongrightarrow{\cong}C(\V_\p,N).$$
In the case $n\geq 1$ there is an up to orientation unique element in $\E_{\p,n}$ fixed by $K_{\p}(\p^{n})$.
Therefore, there are two natural isomorphisms $$C(G_\p/K_\p(\p^{n}),N)\xlongrightarrow{\cong}C(\E_{\p,n},N),$$ which are interchanged by the Atkin-Lehner involution.
\end{Remark}

We will construct a compatible sequence of elements in $\E_{\p,n}$.
Key to this is the following lemma, which is essentially Lemma 6.5 of \cite{CV}. 
\begin{Lemma} \label{neighbours}
Let $v \in \V_\p$ be a vertex of $\T_\p$.
 \begin{enumerate}[(i)]
    \item 	Let $l_\p(v) = 0$.
		\begin{itemize}
		 \item If $\p$ is split in $E$, there are exactly two neighbours $v'$ of $v$ such that $l_\p(v')=0.$
		       They are given by $\varpi_{\P} v$ and $\varpi_{\P}^\tau v$.
		 \item If $\p$ is ramified in $E$, there is exactly one neighbour $v'$ of $v$ such that $l_\p(v')=0.$
		       It is given by $\varpi_{\P} v$. 
		 \item If $\p$ is inert in $E$, there is no such neighbour.
		\end{itemize}
    \item	Let $l_\p(v) \geq 1$.
		Then there exists a unique neighbour $v'$ of $v$ with
		$$l_\p(v') = l_\p(v)-1.$$
    \item	In both cases, $(i)$ and $(ii)$, the remaining neighbours $v'$ of $v$ satisfy $$l_\p(v')=l_\p(v)+1.$$
		They are permuted faithfully and transitively by $U_{T_\p}^{(l_\p(v))}/U_{T_\p}^{(l_\p(v)+1)}$.
 \end{enumerate}
\end{Lemma}

Let $\E_{\p,\infty}=\varprojlim_{n} \E_{\p,n}$ be the set of infinite, non-backtracking sequences of adjacent vertices.
Let $w_0$ be the vertex corresponding to $\mathcal{R}_\p$ or, equivalently, the unique vertex fixed by the action of $K_\p$.
By our assumptions we have $l_\p( w_0)=0$.
Using Lemma \ref{neighbours} (iii) we consecutively choose vertices $ w_i$ such that $ w_i\sim w_{i-1}$ and $l_\p( w_i)=i$ for all $i\geq 1$.
We set $ w_{\infty}=( w_0, w_1, w_2,\dots) \in \E_{\p,\infty}.$
Further, we define $ w_{-1}=\varpi_{\P}  w_0$ if $\p$ is ramified in $E$.
If $\p$ splits in $E$, we set $ w_{-j}=\varpi_{\P}^{-j}  w_0$ for every integer $j > 0$.

It will be convenient to introduce the following notation:
\begin{align*}
 \eta_\p = \begin{cases}
            0	    &\mbox{if $\p$ is inert in $E$,} \\
            -1	    &\mbox{if $\p$ is ramified in $E$,} \\
            -\infty &\mbox{if $\p$ is split in $E$.}
	   \end{cases}
\end{align*}
By definition we have
\begin{align} \label{StabFormel}
 \Stab_{T_\p}( w_{m-n},\dots, w_m) = U_{T_\p}^{(m)}
\end{align}
for all integers $m,n \geq 0$ such that $m-n \geq \eta_\p$.

We define $$\partial_m:= \partial_{( w_{m-n},\dots, w_m)} \colon C^{0}(\E_{\p,n},N) \to \Dist(T_\p/U_{T_\p}^{(m)},N)$$ for all integers $m,n$ as above.

For $m \geq 0$, the projection $\pi_m \colon T_\p/U_{T_\p}^{(m+1)} \to T_\p/U_{T_\p}^{(m)}$ yields maps
\begin{align*}
 (\pi_m)^\ast \colon \Dist(T_\p/U_{T_\p}^{(m)},N) 	& \too \Dist(T_\p/U_{T_\p}^{(m+1)},N))\\
					    f	& \mapstoo f \circ \pi_m
\end{align*}
and
\begin{align*}
 (\pi_m)_\ast \colon \Dist(T_\p/U_{T_\p}^{(m+1)},N) & \too \Dist(T_\p/U_{T_\p}^{(m)},N)\\
					      f	 & \mapstoo \sum_{t\in U_{T_\p}^{(m)}/U_{T_\p}^{(m+1)}} t.f.
\end{align*}

The proof of the following lemma is an easy calculation.
Most of the cases were already dealt with by Cornut and Vatsal in Section 6 of \cite{CV}.
\begin{Lemma}\label{localnorm}
Let $n\geq 0$ be an integer.
The following formulas hold for all $\phi \in C^0(\E_{\p,n},N)$:
\begin{enumerate}[(i)]
 \item For $m \geq \max\left\{1,n+\eta_\p+1\right\}$ the equality $$(\partial_m \circ \mathbb{T}_\p)(\phi) = ((\pi_m)_\ast \circ \partial_{m+1})(\phi) + \cf_{\p}(\p^n) ((\pi_{m-1})^\ast \circ \partial_{m-1})(\phi)$$ holds with $$\cf_{\p}(\p^n)=\begin{cases}1& \mbox{if $n=0$,}  \\ 0 & \mbox{else.}\end{cases}$$
 \item If $n+\eta_\p\leq 0$, the following equality holds:
       $$(\partial_0 \circ \mathbb{T}_\p)(\phi) = ((\pi_0)_\ast \circ \partial_{1})(\phi) + (\ast),$$ 
       where
       \begin{align*}
	(\ast) = \begin{cases}
		  0 												& \mbox{if $\p$ is inert in $E$,} \\
		  \cf_{\p}(\p^n)\varpi_{\P} \partial_0(\phi) 						& \mbox{if $\p$ is ramified in $E$,} \\
		  \cf_{\p}(\p^n)\varpi_{\P} \partial_0(\phi) + (\varpi_{\P})\inv \partial_0(\phi) 	& \mbox{if $\p$ is split in $E$.}
		 \end{cases}
       \end{align*}
 \item If $\p$ is inert in $E$ and $n=1$, then $$(\partial_1 \circ \mathbb{T}_\p \circ W_\p)(\phi)+\partial_1(\phi)	= ((\pi_0)_\ast \circ \partial_1)(\phi)$$ holds.
\end{enumerate}
\end{Lemma}

\begin{Remark}
Let $m\geq 1$ and $n\geq 0$ be integers such that $m-n\geq \eta_p$.
 The only cases where we do not have a formula involving $(\pi_{m-1})_\ast \circ \partial_{m}$ are the following: $n\geq 2$ and either $\p$ is inert in $E$ and $m=n$ or $\p$ is ramified and $m=n-1$.
\end{Remark}

Let us denote by
$$\invol \colon \Dist(T_\p/U_{T_\p^{(m)}},N) \too \Dist(T_\p/U_{T_\p^{(m)}},N)$$
the map induced by inversion. 
The following lemma is the main ingredient for proving a functional equation for Stickelberger elements.
\begin{Lemma}\label{localfunceq}
Let $n \geq 0$ be an integer.
\begin{enumerate}[(i)]
 \item Assume $n \leq -\eta_\p$.
       Then for all $\phi \in C^0(\E_{\p,n},N)$ the equality $$(\partial_0 \circ W_{\p^n})(\phi) = (\invol \circ \partial_0)(J \phi)$$ holds up to multiplication by an element of $T_\p$.
 \item Assume $m \geq n$.
       Then for all $\phi \in C^0(\E_{\p,n},N)$ the equality $$\partial_m(\phi) = (\invol \circ \partial_m)(J \phi)$$ holds up to multiplication by an element of $T_\p$.
\end{enumerate}
\end{Lemma}

\begin{proof}
To prove (i), note that for $t \in T_\p/U_{T_\p}$ we have
\begin{align*}
 \invol (\partial_0(J\phi))(t) &= (J\phi)(t\inv( w_{-n},\dots, w_1, w_0)) \\
				     &= \phi(J\inv t\inv( w_{-n},\dots, w_1, w_0)) \\
				     &= \phi(tJ\inv( w_{-n},\dots, w_1, w_0)) \\
				     &= \phi(t(\varpi_{\P}^n J\inv w_0,\dots,\varpi_{\P} J\inv w_0,J\inv w_0)).
\end{align*}  
Since $t' J\inv  w_0 = J\inv {t'}\inv  w_0 = J\inv  w_0$ holds for all $t' \in U_{T_\p}$ it follows from Lemma \ref{neighbours} that $J\inv  w_0 = \varpi_\P^k  w_0$ for some $k \in \Z$.
This leads to
\begin{align*}
 \invol(\partial_0(J\phi))(t)   &= \phi(t(\varpi_{\P}^n J\inv w_0,\dots,\varpi_{\P} J\inv w_0,J\inv w_0)) \\
				&= \phi(t(\varpi_{\P}^{k+n}  w_0,\dots,\varpi_{\P}^{k-1}  w_0, \varpi_{\P}^k w_0))
\intertext{and we get}
 \varpi_\P^{-k-n} \invol(\partial_0(J\phi))(t)   	&= \phi(t( w_0,\dots,\varpi_{\P}^{-n+1}  w_0, \varpi_{\P}^{-n} w_0)) \\
							&= (W_\p^n \phi) (t ( w_{-n},\dots, w_1, w_0)).
\end{align*}

Claim (ii) follows by a similar calculation as in the first part using that, by Lemma \ref{neighbours}, there exists an element $x \in T_\p$ such that
\begin{align*}
 x(J\inv w_{m-n},\dots,J\inv w_m) = ( w_{m-n},\dots, w_m)
\end{align*}
holds.
\end{proof}
\subsection{Ends and the Steinberg representation} \label{Steinberg}
We will give a quick review of the theory of ends of the Bruhat-Tits tree.
By realizing the Steinberg representation as a space of functions on the set of ends, we construct a map $\delta_\p^{\ast}$ from the dual of the Steinberg representation to the space of distributions on the local torus, which is compatible with the maps $\partial_{m}$ for $m\geq 1$.

We say that two elements $(v_i)_{i \geq 0}$ and $(v'_i)_{i \geq 0}$ in $\E_{\p,\infty}$ are equivalent if there exist integers $N,N' \geq 0$ such that $v_{N+i}=v'_{N'+i}$ for all $i \geq 0$.
An end in $\T_\p$ is defined as an equivalence class of elements in $\E_{\p,\infty}$.
The set of ends is denoted by $\Ends$.
To an edge $e \in \E_\p$ we assign the set $V(e)$ of ends that have a representative containing $e$.
The sets $V(e)$ form a basis of a topology on $\Ends$, which turns $\Ends$ into a compact space.
The natural action of $G_\p$ on $\E_{\p,n}$ extends to an action on $\Ends$.

Let $\mathfrak{F} \subseteq \Ends$ be the set of fix points under the action of $T_\p$.
As a consequence of Lemma \ref{neighbours} we see that $T_\p$ acts simply transitively on the complement of $\mathfrak{F}$.
Hence, choosing a base point $\left[ v_\infty \right]$ in the complement yields a homeomorphism $\kappa_{[v_\infty]} \colon T_\p \to \Ends-\mathfrak{F}$ via $t \mapsto t[v_\infty]$.
In the following we will choose the class of $w_\infty$ as our base point and write $\kappa=\kappa_{[w_\infty]}$.

\begin{Remark}
 The set $\mathfrak{F}$ is non-zero only in the split case, where it consists of two elements given as follows:
 Clearly, the equivalence classes of the elements $$o_\P=( w_0,\varpi_{\P}  w_0,\varpi_{\P}^2  w_0,\dots)\ \mbox{and}\ o_{{\P}^{\tau}}=( w_0,\varpi_\P^\tau  w_0,(\varpi_\P^\tau)^2  w_0,\dots)$$ are fixed by $T_\p$.
 Using Lemma \ref{neighbours} one can show that $\mathfrak{F}=\{[o_\P],[o_{\P^{\tau}}]\}$ holds.
 In particular, the choice of the prime $\P$ lying above $\p$ (and hence, the choice of the vertices $ w_i$ for $i\leq -1$) is equivalent to the choice of the element $[o_\P]$ of $\mathfrak{F}$. 
\end{Remark}

We define the Steinberg representation $\St_\p$ to be the space of locally constant $\Z$-valued functions on $\Ends$ modulo constant functions, i.e.~$\St_\p=C_c^{0}(\Ends,\Z)/\Z$.
The $G_\p$-action on $\Ends$ extends to an action on $\St_\p$ via $(\gamma.\varphi)(\left[v_\infty\right])=\varphi(\gamma\inv \left[v_\infty\right])$ for $\gamma \in G_\p$, $\varphi \in \St_\p$ and $\left[v_\infty\right]\in \Ends$.
The open embedding $\kappa \colon T_\p \into \Ends$ induces a $T_\p$-equivariant map $$\delta_\p \colon C_c^{0}(T_\p,\Z) \too \St_\p$$ and thus, by dualizing we get a map
\begin{align*}
\delta_\p^\ast \colon \Hom(\St_\p,N) \too \Dist(T_\p,N).
\end{align*}
In the split case we can extend $\kappa$ to a map from $F_\p$ to $\Ends$ by mapping $0$ to $o_\P$.
Thus, we can extend $\delta_\p$ to a map 
\begin{align}\label{deltasp}
 \delta_\p \colon C_c^{0}(F_\p,\Z) \too \St_\p, 
\end{align}
which in turn induces a $T_\p$-equivariant map
\begin{align*}
 \delta_\p^\ast \colon \Hom(\St_\p,N) \too \Dist(F_\p,N).
\end{align*}
If $\p$ is non-split, the image of $\kappa$ is equal to $\Ends$.
Therefore, $\delta_\p$ descends to a map
\begin{align}\label{deltansp}
 \delta_\p \colon C_c^{0}(F_\p,\Z)/\Z \too \St_\p
\end{align}
and thus, we have
\begin{align*}
 \delta_\p^\ast \colon \Hom(\St_\p,N) \too \Dist_0(T_\p,N) \subseteq \Dist(T_\p,N).
\end{align*}

Dualizing the canonical map $\E_{\p,1} \to \St_\p$ given by $e \mapsto \cf_{V(e)}$ yields the $G_\p$-equivariant evaluation map
\begin{align}\label{localev}
 \ev_\p \colon \Hom(\St_\p,N) \too C^0(\E_{\p,1},N).
\end{align}
Further, there is the natural map $$\Dist(T_\p,N) \too \Dist(T_\p/U_{T_\p}^{(m)},N)$$ induced by the projection $T_\p \to T_\p/U_{T_\p}^{(m)}$.
By definition we have
\begin{align}\label{compreason}
V(( w_{m-1}, w_m))=\kappa(U_{T_\p}^{(m)})
\end{align}
for all $m\geq 1$ and, if $\p$ is split in $E$, we also have
$$V(( w_{-1}, w_0))=\kappa(\OO_{F_\p}).$$
From this, one easily gets
\begin{Lemma}\label{comp}
\begin{enumerate}[(i)]\thmenumhspace
\item \label{comp1} Let $m \geq 1$ be an integer.
The following diagram is commutative:
\begin{center}
\begin{tikzpicture}
    \path 	(0,0) 	node[name=A]{$\Hom(\St_\p,N)$}
		(5,0) 	node[name=B]{$C^0(\E_{\p,1},N)$}
		(0,-2) 	node[name=C]{$\Dist(T_\p,N)$}
		(5,-2) 	node[name=D]{$\Dist(T_\p/U_{T_\p}^{(m)},N)$};
    \draw[->] (A) -- (B) node[midway, above]{$\ev_\p$};
    \draw[->] (A) -- (C) node[midway, left]{$\delta_\p^\ast$};
    \draw[->] (B) -- (D) node[midway, right]{$\partial_m$};
    \draw[->] (C) -- (D) node[midway, below]{};
\end{tikzpicture} 
\end{center}
\item\label{comp2} Suppose that $\p$ is split in $E$.
Let
$$\alpha_\p^{\ast}\colon \Dist(F_\p,N)\too \Dist(T_\p/U_{T_\p},N)$$
be the dual of the map \eqref{loccompmap}.
Then the following diagram is commutative:
\begin{center}
\begin{tikzpicture}
    \path 	(0,0) 	node[name=A]{$\Hom(\St_\p,N)$}
		(5,0) 	node[name=B]{$C^0(\E_{\p,1},N)$}
		(0,-2) 	node[name=C]{$\Dist(F_\p,N)$}
		(5,-2) 	node[name=D]{$\Dist(T_\p/U_{T_\p},N)$};
    \draw[->] (A) -- (B) node[midway, above]{$\ev_\p$};
    \draw[->] (A) -- (C) node[midway, left]{$\delta_\p^\ast$};
    \draw[->] (B) -- (D) node[midway, right]{$\partial_0$};
    \draw[->] (C) -- (D) node[midway, above]{$\alpha_\p^{\ast}$};
\end{tikzpicture} 
\end{center}
\end{enumerate}
\end{Lemma}

There is also a twisted version of the above constructions if $\p$ is inert in $E$.
Let $\nr\colon B_\p^{\ast}\to F_\p^{\ast}$ denote the reduced norm.
The character $$\chi_{-1}\colon B_\p^{\ast}\too \left\{\pm 1\right\},\ g\mapstoo (-1)^{\ord_{\p}(\nr(g))}$$ is trivial on the center and thus, descents to a character on $G_\p$.
The twisted Steinberg representation is defined by
\begin{align*}
\TSt_\p=\St_\p(\chi_{-1}).
\end{align*}
Since $\p$ is inert in $E$ we have $\ord_{\p}(\nr(t)) \equiv 0 \bmod 2$ for all $t\in T_{\p}$.
Therefore, the map 
\begin{align}\label{deltatw}
\Tdelta\colon C^{0}(T_\p,\Z)/\Z\too \TSt_\p,\ f\mapstoo \delta_\p(f) \otimes 1
\end{align}
is $T_{\p}$-equivariant.

There is also a ($G_{\p}$-equivariant) twisted evaluation map
\begin{align}\label{twistlocalev}
\Tev\colon \Hom(\TSt_\p,N) \too C^0(\E_{\p,1},N).
\end{align}
It is given by dualizing the map $$\E_{\p,1} \too \TSt_\p,\ e \mapstoo \chi_{-1}(g_e)\cdot \cf_{V(e)}\otimes 1,$$ where $g_e\in G_\p$ is any element such that $g_e.( w_0, w_1)=e$.
Again, using \eqref{compreason} we get the following
\begin{Lemma}\label{comptwist}
Let $m \geq 1$ be an integer and let $\p$ be a prime, which is inert in $E$. The following diagram is commutative:
\begin{center}
\begin{tikzpicture}
    \path 	(0,0) 	node[name=A]{$\Hom(\TSt_\p,N)$}
		(5,0) 	node[name=B]{$C^0(\E_{\p,1},N)$}
		(0,-2) 	node[name=C]{$\Dist_0(T_\p,N)$}
		(5,-2) 	node[name=D]{$\Dist(T_\p/U_{T_\p}^{(m)},N)$};
    \draw[->] (A) -- (B) node[midway, above]{$\Tev$};
    \draw[->] (A) -- (C) node[midway, left]{$(\Tdelta)^\ast$};
    \draw[->] (B) -- (D) node[midway, right]{$(-1)^{m+1}\partial_m$};
    \draw[->] (C) -- (D) node[midway, below]{};
\end{tikzpicture} 
\end{center}
\end{Lemma}
\subsection{Global cohomology classes and pullback to the torus} \label{Pullback}
In this section we globalize the constructions of the previous sections.

We fix pairwise disjoint finite sets $S_{\St}$, $S_\tw$ and $S'$ of finite places of $F$ disjoint from $\ram(B)$ and put $S=S_{\St}\cup S_\tw$. 
For an $R$-module $N$ and a compact open subgroup $K \subseteq G(\A^{S,\infty})$ we consider
\begin{align*}
 \mathcal{A}(K,S_{\St},S_{\tw};N)^{S'} = C\Bigl(G(\A^{S \cup S',\infty})/K,\Hom\Bigl(\bigotimes_{\p\in S_{\St}}\St_{\p} \otimes \bigotimes_{\p\in S_{\tw}}\TSt_{\p},N\Bigr)\Bigr)
\end{align*}
with its natural $G(F)$-action, i.e.~for every $\p\in S_{\St}$ (resp.~$\p\in S_{\tw}$) we view $\St_\p$ (resp. $\TSt_{\p}$) as a $G(F)$-module via the embedding $G(F)\into G_\p$ and put
$$(g.\Phi)(x)(f_{\St}\otimes f _{\tw})=\Phi(g^{-1}x)((g^{-1}f_{\St})\otimes (g^{-1}f _{\tw}))$$
for $g\in G(F)$, $\Phi\in \mathcal{A}(K,S_{\St},S_{\tw};N)^{S'}$, $x\in G(\A^{S \cup S',\infty})/K$, $f_{\St}\in \bigotimes_{\p\in S_{\St}}\St_{\p}$ and $f_{\tw}\in \bigotimes_{\p\in S_{\tw}}\TSt_{\p}$.
Further, we fix a locally constant character $$\epsilon\colon T_{\infty}\too \left\{\pm 1\right\}.$$
We will often view $\epsilon$ as a character on $T(F)$ via the embedding $T(F)\into T_\infty$.
There exists a unique extension $\epsilon\colon G_\infty\to \left\{\pm 1\right\}$ such that the diagram
\begin{center}
\begin{tikzpicture}
    \path 	(0,0) 	node[name=A]{$T_\infty$}
		(2,0) 	node[name=B]{$\left\{\pm 1\right\}$}
		(0,-1) 	node[name=C]{$G_\infty$};
    \draw[->] (A) -- (B) node[midway, above]{$\epsilon$};
    \draw[->] (C) -- (B) node[midway, below]{$\epsilon$};
    \draw[->] (A) -- (C) node[midway, left]{$\iota$};
\end{tikzpicture} 
\end{center}
is commutative.
Again, we view $\epsilon$ also as a character on $G(F)$ via the embedding $G(F)\into G_\infty$.
\begin{Def}
  The space of $N$-valued, $(S_{\St},S_\tw)$-special modular symbols on $G$ of level $K$ and sign $\epsilon$ is defined to be
\begin{align*}
 \mathcal{M}(K,S_{\St},S_\tw;N)^{\epsilon} = \HH^d(G(F),\mathcal{A}(K,S_{\St},S_\tw;N)(\epsilon)).
\end{align*}
\end{Def}

Let $\n \subseteq \OO_F$ be a non-zero ideal coprime to $\ram(B)$.
We fix an Eichler order $\mathcal{R}(\n) \subseteq \mathcal{R}$ of level $\n$ contained in the fixed maximal order $\mathcal{R}$.
As in the local case, we write $K_{\p}$ (resp.~$K_{\p}(\n)$) for the image of $\mathcal{R}_{\p}^{\ast}$ (resp.~ $\mathcal{R}(\n)^{\ast}_\p$) in $G_{\p}$ and set
\begin{align*}
 K=\prod_{\p \notin S_{\infty}} K_\p\quad \left(\mbox{resp.}\ K(\n)=\prod_{\p \notin S_{\infty}} K_\p(\n)\right).
\end{align*}
We put
$$\mathcal{M}(\n,S_\St,S_{\tw};N) = \mathcal{M}(K(\n)^{S},S_\St,S_\tw;N)$$
and
$$\mathcal{M}(\n;N)=\mathcal{M}(\n,\emptyset,\emptyset;N).$$
Without loss of generality we will always assume that every $\p\in S$ divides $\n$ exactly once.

For an open subgroup $\tild{U} \subseteq U_T^{S \cup S',\infty}$ we define $$\mathcal{D}(\tild{U},S;N)^{S',\infty}=\Hom_R(\mathcal{C}_c(\tild{U},S,R)^{S',\infty},N).$$
In case $\tild{U}=U_T(\m)$ with $\m \subseteq \OO_F$ a non-zero ideal we write $\mathcal{D}(\m,S;N)^{S',\infty}$ for the corresponding distribution space.

From now on we assume that every prime $\p$ in $S_{\tw}$ is inert in $E$.
Thus, the local maps \eqref{deltasp} and \eqref{deltansp} (resp.~\eqref{deltatw}) induce the semi-local map
\begin{align*}
 \delta_{S_\St} = \otimes_{\p \in S_\St} \delta_\p \colon \bigotimes_{\substack{\p\in S_\St,\\ \p\ \text{split}}} C_c(F_\p,\Z)\otimes \hspace{-1em}\bigotimes_{\substack{\p\in S_\St,\\ \p\ \text{non-split}}}\hspace{-1em}C_c(T_\p,\Z)/\Z \too \bigotimes_{\p \in S_\St}\St_{\p}
\end{align*}
respectively
\begin{align*}
 \delta^{\tw}_{S_\tw} = \otimes_{\p \in S_\tw} \Tdelta\colon \bigotimes_{\p\in S_\tw}C_c(T_\p,\Z)/\Z \too \bigotimes_{\p \in S_\tw}\TSt_{\p}.
\end{align*}

For every compact open subgroup $K \subseteq G(\A^{S \cup S',\infty})$ and every $g \in G(\A^{S \cup S',\infty})$ we get a $T(F)$-equivariant homomorphism
\begin{align*}
 \Delta_{g,S_\St,S_\tw}^{S'} \colon \mathcal{A}(K,S_\St,S_\tw;N)^{S'} \too \mathcal{D}(\iota^{-1}(gKg^{-1}),S;N)^{S'}
\end{align*}
via
\begin{align*}
 \Delta_{g,S_\St,S_\tw}^{S'}(\Phi)(x)(f_{S_\St}\otimes f _{S_\tw}) =\Phi(\iota(x)g)(\delta_{S_\St}(f_{S_\St})\otimes\delta^{\tw}_{S_\tw}(f_{S_\tw}))
\end{align*}
for $x \in T(\A^{S \cup S',\infty})/\iota^{-1}(gKg^{-1})$ and $f_{S_\St}$, as well as $f_{S_\tw}$, in the appropriate semi-local function spaces.

Composing $\Delta_{g,S_\St,S_\tw}$ with the restriction map
\begin{align*}
 \mathcal{M}(K,S_\St,S_\tw;N)^{\epsilon} &\too \HH^d(T(F),\mathcal{A}(K,S_\St,S_\tw;N)(\epsilon))
\intertext{on cohomology yields a map}
 \mathcal{M}(K,S_\St,S_\tw;N)^{\epsilon} &\too \HH^d(T(F),\mathcal{D}(\iota^{-1}(gKg^{-1}),S;N)^{\infty}(\epsilon)),
\end{align*}
which we will also denote by $\Delta_{g,S_\St,S_\tw}$.

Note that by Remark \ref{EichlerTree} there is an up to orientation unique $G(\A^{\ram(B)\cup S,\infty})$-equivariant isomorphism 
\begin{align} \label{Ecken}
 G(\A^{\ram(B)\cup S,\infty})/K(\n)^{\ram(B)\cup S} \cong \prod_{\p \not\in \ram(B)\cup S\cup S_{\infty}}\nolimits' \E_{\p,\ord_\p(\n)}.
\end{align}
\begin{Def}\label{allowable}
A non-zero ideal $\m \subseteq \OO_F$ is called $\n$-allowable if $\m$ is coprime to $\ram(B)$ and $\ord_\p(\m) - \ord_\p(\n) > \eta_\p$ for all $\p \notin\ram(B)$.
\end{Def}
Let us fix an $\n$-allowable ideal $\m$.
For a finite place $\p$ of $F$ that is not in $S\cup\ram(B)$ we define $e_\p=(w_{\ord_\p(\m)-\ord_\p(\n)},\dots,w_{\ord_\p(\m)})$, where the $w_i$ are the vertices chosen in Section \ref{Local}.
Let $g_\m = (g_\p)_\p \in G(\A^{S,\infty})/K(\n)^{S}$ be the element that is equal to one at places in $\ram(B)$ and corresponds to $(e_\p)_\p$ under the above isomorphism for all places $\p\notin S\cup \ram(B)$.
In this case, the equality
$$U_T(\m)=\iota^{-1}(g_\m K(\n) g^{-1}_\m)$$
holds and hence, we have a map
\begin{align*}
 \Delta_{\m,S_\St,S_\tw}= \Delta_{g_\m,S_\St,S_\tw} \colon \mathcal{M}(\n,S_\St,S_\tw;N)^{\epsilon} \too \HH^d(T(F),\mathcal{D}(\m,S;N)^{\infty}(\epsilon)).
\end{align*}
As always, we drop $S_\St$ and $S_\tw$ from the notation if they are empty.

For every $\p\notin \ram(B)$ the Hecke operator $\mathbb{T}_\p$ as defined in \eqref{Heckeoperatoren} acts on $\mathcal{M}(\n;N)^{\epsilon}$ via the isomorphism \eqref{Ecken}.
Similarly, for $\n' \mid \n$ the global Atkin-Lehner involution $W_{\n'}$ is given by applying the local Atkin-Lehner involutions $W_{\p^{\ord_\p(\n')}}$ at the places $\p \mid \n'$.
Also, for every $\p\in\ram(B)$ the local Atkin-Lehner involution $W_\p$ is given by interchanging the two elements in the set $G_\p/K_\p$.
\subsection{Anticyclotomic Stickelberger elements} \label{Stickelberger}
We are going to define anticyclotomic Stickelberger elements, bound their order of vanishing from below and prove a functional equation.
Throughout this section we fix a ring $R$, an $R$-module $N$, a non-zero ideal $\mathfrak{n}\subseteq\mathcal{O}$, which is coprime to $\ram(B)$, and a character $\epsilon$ as before.
In addition, we fix a modular symbol $\kappa\in \mathcal{M}(\mathfrak{n};N)^{\epsilon}$. 
Stickelberger elements will be defined by taking cap products of various pullbacks of $\kappa$ with the homology class defined in Section \ref{Vanishing} associated to the Artin reciprocity map.

\begin{Def}
A finite Abelian extension $L$ over $E$ is called anticyclotomic if it is Galois over $F$ and $\tau\sigma\tau\inv=\sigma\inv$ holds for all $\sigma\in \Gal(L/E)$.
\end{Def}

We fix an anticyclotomic extension $L/E$ with Galois group $\G=\G_{L/E}$.
The Artin reciprocity map induces a group homomorphism
\begin{align*}
 \rec_{L/E}\colon T(\A)/T(F)\too \G.
\end{align*}
In addition, we fix an $\n$-allowable ideal $\m$ of $\OO_F$ that bounds the ramification of $L/E$, i.e.~$U_T(\m)$ is contained in the kernel of $\rec_{L/E}$.
Let $$c_L=c_{\rec_{L/E}}\in \HH_d(T(F),\mathcal{C}(\m,\Z[\G])^{\infty}(\epsilon))$$ be the image of $\rec_{L/E}$ under \eqref{homclass} with $S=\emptyset$. We adopt similar notations if $S$ is not the empty set, e.g.~we set $c_{L}(\m,S,\epsilon)=c_{\rec_{L/E}}(\m,S,\epsilon)$.

The natural pairing
\begin{align*}
 \mathcal{C}_{c}(\mathfrak{m},\Z[\G])^{\infty}\times \mathcal{D}(\mathfrak{m};N)^\infty\too \Z[\G]\otimes N
\end{align*}
induces a cap-product pairing
\begin{align*}
 \HH_{d}(T(F),\mathcal{C}_{c}(\mathfrak{m},\Z[\G])^{\infty}(\epsilon))\times \HH^{d}(T(F),\mathcal{D}(\mathfrak{m};N)^\infty(\epsilon))\too \Z[\G]\otimes N.
\end{align*}

\begin{Def}\label{StickelDef}
The anticyclotomic Stickelberger element of modulus $\mathfrak{m}$ associated with $\kappa$ and $L/F$ is defined as the cap-product
\begin{align*}
 \Theta_{\mathfrak{m}}(L/F,\kappa)=\Delta_{\mathfrak{m}}(\kappa)\cap c_L \in \Z[\G]\otimes N.
\end{align*}
\end{Def}

As a direct consequence of functoriality of the Artin reciprocity map we get the following compatibility property:

\begin{Proposition}
Let $L^\prime$ be an intermediate extension of $L/E$.
Then we have
\begin{align*}
 \pi_{L/L^\prime}(\Theta_{\mathfrak{m}}(L/F,\kappa))=\Theta_{\mathfrak{m}}(L^\prime/F,\kappa),
\end{align*}
where
\begin{align*}
 \pi_{L/L^{\prime}}\colon \Z[\mathcal{G}_{L/E}]\otimes N\too \Z[\mathcal{G}_{L^\prime/E}]\otimes N
\end{align*}
is the canonical projection.
\end{Proposition}

Let $k$ be an $R$-algebra and $\chi\colon\G\to k^{\ast}$ a character.
Via the Artin reciprocity map we can view $\chi$ as a character of $T(\A)$.
The character also induces an $R$-linear map $\chi\colon \Z[\G]\otimes N\too k\otimes_R N$.
Orthogonality of characters immediately implies the following result. See \cite{BG}, Proposition 1.12, for a more detailed proof.

\begin{Proposition}
Let $k$ be an $R$-algebra which is a field and let $\chi\colon\G\to k^{\ast}$ be a character.
If $\chi_\infty\neq \epsilon$, we have $$\chi(\Theta_{\mathfrak{m}}(L/F,\kappa))=0.$$
\end{Proposition} 

Let $S_\St$ and $S_\tw$ be finite disjoint sets of finite places of $F$ with
\begin{itemize}
 \item $\mathfrak{p}$ divides $\mathfrak{n}$ exactly once for all $\mathfrak{p}\in S=S_\St\cup S_\tw$, 
 \item $S$ is disjoint from $\ram(B)$ and
 \item every prime in $S_\tw$ is inert in $E$.
\end{itemize}
The local evaluation maps \eqref{localev} and \eqref{twistlocalev} induce a map $$\Ev_{S_\St,S_\tw} \colon \mathcal{M}(\n,S_\St,S_\tw;N)^{\epsilon}\too \mathcal{M}(\n;N)^{\epsilon}.$$
For a place $v$ of $F$ we let $\mathcal{G}_{v} \subseteq \mathcal{G}$ be the decomposition group at $v$.
If $\p\in S$, we define $I_{\p} \subseteq \Z[\mathcal{G}]$ as the kernel of the projection $\Z[\mathcal{G}]\onto \Z[\mathcal{G}/\mathcal{G}_{\p}]$.
If $v\in S_{\infty}$ is split in $E$, we let $\sigma_{v}$ be a generator of $\mathcal{G}_{v}$ and define $I_{v}^{\pm 1} \subseteq \Z[\mathcal{G}]$ as the ideal generated by $\sigma_{v}\mp 1$.
For non-split Archimedean places we define $I_{v}^{\pm}=\Z[\mathcal{G}]$.

\begin{Lemma}\label{vanish2}
Assume that $N$ is $\Z$-flat and that there exists an $(S_\St,S_\tw)$-special modular symbol $\kappa'\in \mathcal{M}(\n,S_\St,S_\tw,;N)^\epsilon$ lifting $\kappa$, i.e. $\Ev_{S_\St,S_\tw}(\kappa_S)=\kappa$ holds.
Then we have
\begin{align*}
 \Theta_{\mathfrak{m}}(L/F,\kappa)\in \left(\prod_{v\in S_{\infty}} I_{v}^{-\epsilon_{v}(-1)} \cdot \prod_{\p \in S_\mathfrak{m}} I_{\p}\right) \otimes N.
\end{align*}
In particular, if $N=R$ is a $\Z$-flat ring  and $\epsilon$ is trivial, we have
\begin{align*}
 2^{-d}\Theta_{\mathfrak{m}}(L/F,\kappa)&\in R[\mathcal{G}]
\intertext{and}
 \ord_{R}(2^{-d}\Theta_{\mathfrak{m}}(L/F,\kappa))&\geq \left|S_{\m}\right|.
\end{align*}
\end{Lemma}

\begin{proof}
By Lemma \ref{comp} \eqref{comp1} and Lemma \ref{comptwist} we have
\begin{align*}
 \Theta_{\m}(L/F,\kappa) & = \Delta_{\m}(\kappa)\cap c_{L}(\m,\emptyset,\epsilon)   \\
			 & =\pm\ \Delta_{\m,S_\St,S_\tw}(\kappa')\cap c_{L}(\m,S,\epsilon).
\end{align*}
We set $I=\prod_{v\in S_{\infty}} I_{v}^{-\epsilon_{v}(-1)} \cdot \prod_{\p \in S_\mathfrak{m}} I_{\p}$ and consider the ring $A=\Z[\mathcal{G}]/I$ together with the projection maps $\pi\colon \Z[\G]\to A$ and $\pi_N\colon \Z[\G]\otimes N \to A\otimes N$.
We have
\begin{align*}
 \pi_N (\Theta_{\m}(L/F,\kappa))=\pm\ \Delta_{\m,S_\St,S_\tw}(\kappa')\cap\pi_{\ast}(c_{L}(\m,S,\epsilon))=0
\end{align*}
since the homology class $\pi_{\ast}(c_{L}(\m,S,\epsilon))=c_{\pi\circ \rec_{L/E}}(\mathfrak{m},S,\epsilon)$ vanishes by applying Proposition \ref{vanish1} with $\mathfrak{a}=A$.
\end{proof}

\begin{Lemma}\label{funceq}
Suppose that every $\p\in\ram(B)$ is inert in $E$ and that we can decompose $\n=\n_1 \n_2$ such that $\n_1$ is coprime to $\m$ and $\n_2 \mid \m$.
Write $\n_1 = \prod_{i=1}^r \p_i^{n_i}$, with $n_i \geq 1$ for $1 \leq i \leq r$.
Let $\kappa$ be an eigenvector of $W_{\p_i^{n_i}}$ with eigenvalue $\varepsilon_{i} \in \{ \pm 1\}$ for $1 \leq i \leq r$ and of $W_\p$ with eigenvalue $\varepsilon_{\p}\in \{ \pm 1\}$ for every $\p\in\ram(B)$.
Further, write $\varepsilon_{\n_1} = \prod_{i=1}^r\varepsilon_i$ for the eigenvalue of $W_{\n_1}$.
Then
\begin{align*}
 \Theta_\m(L/F,\kappa)^\vee = (-1)^d \cdot \epsilon(-1) \cdot \varepsilon_{\n_1} \prod_{\p\in\ram(B)}\varepsilon_{\p} \cdot \Theta_\m(L/F,\kappa)
\end{align*}
holds up to multiplication with an element in $\G$.
\end{Lemma}

\begin{proof}
This follows directly from Lemma \ref{localfunceq}. See \cite{BG}, Proposition 1.15, for more details.
\end{proof}

\section{Automorphic forms}
We will apply the results of the previous section to cohomology classes coming from automorphic forms.
To this end, let $\pi$ be a cuspidal automorphic representation of $\PGL_2(\A)$ with the following properties:
\begin{itemize}
 \item $\pi_v$ is a discrete series representation of weight $2$ for all Archimedean places $v$ of $F$ and
 \item $\pi_\p$ is special, i.e.~a twist of the Steinberg representation, for all $\p \in \ram(B)$.
\end{itemize}
We write $\Gamma_0(\n)\subseteq \PGL_2(\A)$ for the usual adelic congruence subgroup of level $\n$.
By the automorphic formulation of Atkin-Lehner theory due to Casselman (see \cite{Ca}) there exists a unique non-zero ideal $\mathfrak{f}(\pi) \subseteq \OO_F$ such that $(\pi^\infty)^{\Gamma_{0}(\mathfrak{f}(\pi))}$ is one-dimensional.
Thus, the standard Hecke operator $\mathbb{T}_\p$ (resp.~the Atkin-Lehner involutions $\textnormal{W}_\p$) acts on $(\pi^\infty)^{\Gamma_{0}(\mathfrak{f}(\pi))}$ via multiplication by a scalar which we denote by $\lambda_\p$ (resp. $\omega_\p$).
A result of Clozel (cf.~\cite{C}) tells us that there exists a smallest subfield $\Q_\pi \subseteq \C$, which is a finite extension of $\Q$, such that $\pi^{\infty}$ can be defined over $\Q_\pi$.
More precisely, the Hecke eigenvalues $\lambda_\p$ are elements of the ring of integers $R_\pi$ of $\Q_\pi$.
\subsection{Stickelberger elements associated to automorphic representations} \label{JacquetLanglands}
Under our assumptions on $\pi$, Jacquet and Langlands have proven in \cite{JL} that there exists a transfer of $\pi$ to $B$, i.e.~there exists an automorphic representation $\pi_B$ of $G(\A)$ such that 
\begin{itemize}
\item $\pi_{B,v}\cong \pi_v$ for all places $v$ at which $B$ is split,
\item $\pi_{B,v}$ is the trivial one-dimensional representation for all $v\in S_\infty$ at which $B$ is non-split and
\item $\pi_{B,\p}$ is the trivial (resp.~non-trivial) smooth one-dimensional representation of $G_\p$ for every $\p\in\ram(B)$ for which $\pi_\p$ is the (twisted) Steinberg representation. In particular, the eigenvalue of $W_\p$ acting on $\pi_{B,\p}$ is the negative of the root number of $\pi_\p$.
\end{itemize}
Let $\mathfrak{f}(\pi_{B})$ be the maximal divisor of $\mathfrak{f}(\pi)$ which is coprime to $\ram(B)$.
We define $$\mathcal{M}(\mathfrak{f}(\pi_B);\Q_\pi)^{\epsilon,\pi} \subseteq \mathcal{M}(\mathfrak{f}(\pi_B);\Q_\pi)^{\epsilon}$$ to be the common eigenspace of the operators $\mathbb{T}_\p$ for $\p\notin\ram(B)$ with eigenvalues $\lambda_\p$.
The formalism of $(\mathfrak{g},K)$-cohomology together with the strong multiplicity one theorem implies that $\mathcal{M}(\mathfrak{f}(\pi_B);\Q_\pi)^{\epsilon,\pi}$ is one-dimensional for every sign character $\epsilon$.
By Theorem 11.4.4 of \cite{BS} every arithmetic group $\Gamma$ is of type (VFL) and therefore, the functor $N\mapsto H^\ast(\Gamma,N)$ commutes with direct limits (cf.~\cite{Se2}, p.~101).
It follows that the canonical map
$$\mathcal{M}(\mathfrak{f}(\pi_B);R_\pi)\otimes \Q_\pi\too \mathcal{M}(\mathfrak{f}(\pi_B);\Q_\pi)$$
is an isomorphism.
Therefore, the intersection of $\mathcal{M}(\mathfrak{f}(\pi_B);\Q_\pi)^{\epsilon,\pi}$ with the image of $\mathcal{M}(\mathfrak{f}(\pi_B);R_\pi)$ in $\mathcal{M}(\mathfrak{f}(\pi_B);\Q_\pi)$ is a locally free $R_\pi$-module of rank one.
We choose a maximal element $\kappa^{\pi_B,\epsilon}$ of this module.

\begin{Remark}
\begin{enumerate}[(i)]\thmenumhspace
 \item If $R_\pi$ is a PID, the generator $\kappa^{\pi_B,\epsilon}$ is unique up to multiplication by an element in $R_\pi^{\ast}$.
       In particular, if the automorphic representation $\pi$ corresponds to a modular elliptic curve over $F$, then $\Q_\pi$ is equal to $\Q$ and thus, $\kappa^{\pi_B,\epsilon}$ is unique up to sign.
 \item We could weaken the assumptions on $\pi_\p$ for $\p \in \ram(B)$.
       It is enough to assume that $\pi_\p$ is either special or supercuspidal.
       But in the supercuspidal case there is no canonical local new vector for $\pi_{B,\p}$.
       To ease the exposition, we stick to the special case.
\end{enumerate}
\end{Remark}

Let $L/E$ be a finite anticyclotomic extension with Galois group $\mathcal{G}$ and let $\m$ be an $\mathfrak{f}(\pi_B)$-allowable ideal of $\OO_F$ that bounds the ramification of $L/E$.
\begin{Def}
The anticyclotomic Stickelberger element of modulus $\mathfrak{m}$ and sign $\epsilon$ associated to $\pi_B$ and $L/F$ is defined by $$\Theta_{\mathfrak{m}}(L/F,\pi_B)^{\epsilon}= \Theta_{\mathfrak{m}}(L/F,\kappa^{\pi_B,\epsilon}) \in R_\pi[\mathcal{G}].$$
\end{Def}

\begin{Remark}
The element $\Theta_{\mathfrak{m}}(L/F,\pi_B)^{\epsilon}$ depends on the choice of an $U_{T_{\p}}$-stable vertex and an end of the Bruhat-Tits tree for every prime $\p\notin\ram(B)$. If we take different choices, $\Theta_{\mathfrak{m}}(L/F,\pi_B)^{\epsilon}$ is multiplied by an element of $\G$.
Therefore, the element
$$\mathfrak{L}_{\mathfrak{m}}(L/F,\pi_B)^{\epsilon}=\Theta_{\mathfrak{m}}(L/F,\pi_B)^{\epsilon}\cdot (\Theta_{\mathfrak{m}}(L/F,\pi_B)^{\epsilon})^{\vee}\in R_{\pi}[\G]$$
is independent of these choices.
\end{Remark}

Next, we study the behaviour of Stickelberger elements under change of modulus.

\begin{Theorem}[Norm relations] \label{compatibility}
\begin{enumerate}[(i)]\thmenumhspace
 \item Let $\mathfrak{p}$ be a finite place of $F$ that does not divide $\mathfrak{m}$.
       Write $\sigma_{\P}$ for the the image of the uniformizer $\varpi_{\P}$ under the Artin reciprocity map $\rec_{L/E}$.
       Then the equality
       \begin{align*}
	\Theta_{\mathfrak{m}\mathfrak{p}}(L/F,\pi_B)^{\epsilon} = (\lambda_{\mathfrak{p}}-(\ast))\Theta_{\mathfrak{m}}(L/F,\pi_B)^{\epsilon}
       \end{align*}
holds with
       \begin{align*}
	(\ast) = \begin{cases}
		  0 								& \mbox{if $\p$ is inert in $E$,} \\
		  \cf_{\p}(\mathfrak{f}(\pi_{B}))\sigma_{\P}  		& \mbox{if $\p$ is ramified in $E$,} \\
		  \sigma_{\P}\inv + \cf_{\p}(\mathfrak{f}(\pi_{B}))\sigma_{\P} 	& \mbox{if $\p$ is split in $E$.}
		 \end{cases}
	\end{align*}
 \item Let $\mathfrak{p}$ be a finite place of $F$ that does divide $\mathfrak{m}$ and write $m=\ord_{\p}(\m)$.
       Then we have a decomposition
       \begin{align*}
	\Theta_{\mathfrak{m}\mathfrak{p}}(L/F,\pi_B)^{\epsilon} =\lambda_{\mathfrak{p}}\Theta_{\mathfrak{m}}(L/F,\pi_B)^{\epsilon}\ +\ \cf_{\p}(\mathfrak{f}(\pi_{B})) v_{\mathfrak{m}}(\Theta_{\m\p\inv}(L/F,\pi_B)^{\epsilon}),
       \end{align*}
       where the elements $v_{\mathfrak{m}}(\Theta_{\m\p\inv}(L/F,\pi_B)^{\epsilon})$ can be characterized by the following properties:
       \begin{itemize}
	\item $\pi_{L/L^{\prime}}(v_{\mathfrak{m}}(\Theta_{\m\p\inv}(L/F,\pi_B)^{\epsilon})=v_{\mathfrak{m}}(\Theta_{\m\p\inv}(L^{\prime}/F,\pi_B)^{\epsilon})$ for all intermediate extensions $L^{\prime}$ of $L/F$
	\item $v_{\mathfrak{m}}(\Theta_{\m\p\inv}(L/F,\pi_B)^{\epsilon})= [U_{T_\p}^{(m-1)}:U_{T_\p}^{(m)}](\Theta_{\m\p\inv}(L/F,\pi_B)^{\epsilon})$ in case the Artin reciprocity map for $L/E$ is trivial on $U_T(\m\p\inv)$
	\item Let $k$ be a field which is an $R_\pi$-algebra and $\chi\colon \mathcal{G}\to k^{\ast}$ a character such that $\chi_{\p}$ has conductor $\mathfrak{p}^{m}$.
	      Then we have
	      \begin{align*}
	       \chi(v_{\mathfrak{m}}(\Theta_{\m\p\inv}(L/F,\pi_B)^{\epsilon}))=0.
	      \end{align*}
       \end{itemize}
       \item Suppose that $\p$ is inert and divides $\m$ as well as $\mathfrak{f}(\pi_{B})$ exactly once.
	     Let $k$ be an $R_\pi$-algebra and $\chi\colon \G\to k^{\ast}$ a character which is unramified at $\p$.
	     Then we have
	     $$\chi(\Theta_{\mathfrak{m}}(L/F,\pi_B)^{\epsilon})=0.$$
\end{enumerate}
\end{Theorem}

\begin{proof}
This is a direct consequences of the local norm relations of Lemma \ref{localnorm}.
For part (iii), note that the local representation at $\pi_\p$ is a (twisted) Steinberg representation and thus, the eigenvalue of $\mathbb{T}_\p\circ W_\p$ on a local new vector is $-1$.
\end{proof}

In the following we use the same notation as in the discussion before Lemma \ref{vanish2}.
Let $S_\St$ (resp.~$S_\tw$) be the set of finite places $\p$ of $F$ which are disjoint from $\ram(B)$ (and inert in $E$) such that the local component $\pi_\p$ is the (twisted) Steinberg representation.
As always, we set $S=S_\St\cup S_\tw$.
For every subset $\mathfrak{S}\subseteq S$ let $t_{\mathfrak{S}}\in \Z$ be the product of the exponent of the $2$-torsion subgroup of $\mathcal{M}(\n;R_{\pi})^{\epsilon}$ and the exponent of the torsion subgroup of
\begin{align*}
\bigoplus_{\p\in \mathfrak{S}} \mathcal{M}(\n\p^{-1};R_{\pi})^{\epsilon}.
\end{align*}
In the case $d=0$ the above cohomology groups are torsion-free and hence, $t_{\mathfrak{S}}=1$. 
If $d>0$, we define $c_{\mathfrak{S}}=\gcd\left\{\prod_{\p\in \mathfrak{S}^{\prime}}(N(\p)+1)\mid \mathfrak{S}^{\prime}\subset \mathfrak{S} \mbox{ with } \left|\mathfrak{S}\right|=\left|\mathfrak{S}^{\prime}\right|+1\right\}$. For $d=0$ we simply put $c_{\mathfrak{S}}=1$. Finally, we define $n_{\mathfrak{S}}=c_{\mathfrak{S}}\cdot\ t_{\mathfrak{S}}$.

\begin{Theorem}[Order of vanishing] \label{ordvanish}
For every anticyclotomic extension $L/E$ and every $\mathfrak{f}(\pi_B)$-allowable modulus $\m$ that bounds the ramification of $L/E$ we have
$$n_{S_\m}\Theta_{\mathfrak{m}}(L/F,\pi_B)^{\epsilon}\in \left(\prod_{v\in S_{\infty}} I_{v}^{-\epsilon_{v}(-1)} \cdot \prod_{\p \in S_\mathfrak{m}} I_{\p}\right)\otimes R_{\pi}.$$
\end{Theorem}

\begin{proof}
It is easy to see that the map 
\begin{align*}
 \Ev_{S_\St,S_\tw} \colon \mathcal{M}(\mathfrak{f}(\pi_{B}),S_\St,S_\tw;\Q_\pi)^{\epsilon,\pi}\too \mathcal{M}(\mathfrak{f}(\pi_{B});\Q_\pi)^{\epsilon,\pi}
\end{align*}
is an isomorphism of one-dimensional $\Q_\pi$-vector spaces (cf. \cite{Sp}, Proposition 5.8, for a proof in the Hilbert modular setting).
Therefore, the map
\begin{align*}
 \mathcal{M}(\mathfrak{f}(\pi_{B}),S_\St,S_\tw;R_\pi)^{\epsilon,\pi}\too \mathcal{M}(\mathfrak{f}(\pi_{B});R_\pi)^{\epsilon,\pi}
\end{align*}
has finite cokernel. As in \cite{BG}, Lemma 2.9, one can show that $n_{S_\m}$ annihilates this cokernel.
Therefore, the claim is a direct consequence of Lemma \ref{vanish2}.
\end{proof}

As a direct consequence of Lemma \ref{funceq} we get the following

\begin{Proposition}[Functional equation] \label{globfunceq}
Suppose that every $\p\in\ram(B)$ is inert in $E$ and that we can decompose $\mathfrak{f}(\pi_{B})=\n_1 \n_2$ with $\n_1$ coprime to $\m$ and $\n_2 \mid \m$.
Let $\varepsilon$ be the root number of $\pi$ and $\varepsilon_{\n_2}$ the product of the local root numbers of primes dividing $\n_2$.
Then the equality
\begin{align*}
 (\Theta_\m(L/F,\pi_B)^\epsilon)^\vee = \epsilon(-1) \cdot \varepsilon \cdot\varepsilon_{n_2} \cdot \Theta_\m(L/F,\pi_B)^\epsilon
\end{align*}
holds up to multiplication with an element in $\G$.
\end{Proposition}

\begin{Corollary}[Parity]\label{parity}
Suppose that every $\p\in\ram(B)$ is unramified in $E$ and that there is a decomposition $\mathfrak{f}(\pi_{B})=\n_1 \n_2$ with $\n_1$ coprime to $\m$ and $\n_2 \mid \m$. If $\ord_{R_\pi[\frac{1}{2}]}(\Theta_{\mathfrak{m}}(L/F,\pi_B))=r<\infty$ holds, we have
\begin{align*}
(-1)^{r}=\epsilon(-1)\cdot\varepsilon \cdot\varepsilon_{n_2}.
\end{align*}
\end{Corollary}
\begin{proof}
This follows from the fact that inversion acts as multiplication by $(-1)^{r}$ on $I_{R_\pi[\frac{1}{2}]}^r/I_{R_\pi[\frac{1}{2}]}^{r+1}$.
\end{proof}

\subsection{Interpolation formulae} \label{Interpolation}
We relate anticyclotomic Stickelberger elements to special values of $L$-functions. The crucial input is a computation of toric period integrals by File, Martin and Pitale (cf.~\cite{FMP}).
We keep the notations from the previous section.

Let $\chi_{E/F}\colon \Gal(E/F)\to \C^{\ast}$ be the non-trivial character. 
Given a character $\chi\colon\G\to\C^{\ast}$ and a finite place $\p$ of $F$ we denote by $\varepsilon(1/2,\pi_{E,\p}\otimes \chi_\p)$ the local epsilon factor of the base change of $\pi$ to $\PGL_2(E)$ twisted by $\chi$.
Here we view characters $\chi$ as characters on $T(\A)$ via the Artin reciprocity map.
We say that $\chi$ fulfills the Saito-Tunnell condition with respect to $B$ if for all finite places $\p$ of $F$ the following equality holds:
$$\varepsilon(1/2,\pi_{E,\p}\otimes \chi_\p)= \chi_{E/F,\p}(-1)\INV(B_\p)$$
Here $\INV(B_\p)\in\left\{\pm 1\right\}$ denotes the local invariant of $B$ at $\p$. By our assumptions on the splitting behaviour of $B$ there is no condition at the Archimedean places. 

Given any automorphic representation $\widetilde{\pi}$ of a reductive algebraic group over $F$ and a finite set $S$ of places of $F$ we write
$L^{S}(s,\widetilde{\pi})$ for the $L$-function without the Euler factors at places in $S$ and $L_{S}(s,\widetilde{\pi})$ for the product of the Euler factors of places in $S$.

Let $S(\pi)$ be the set of finite places at which $\pi$ is ramified. For a character $\chi$ as above we set $S(\chi)$ to be the set of finite places at which $\chi$ is ramified. Finally, let $\Sigma(\pi,\chi)$ be the set of all finite places $\p$ such that either the local conductor of $\pi$ at $\p$ is greater than one or the local conductor of $\pi$ at $\p$ is exactly one, $E/F$ is ramified at $\p$ and $\chi_\p$ is unramified. The ramification index of $E/F$ at a prime $\p$ will be denoted by $e_{\p}(E_\p/F_\p)$.

\begin{Theorem}
There exists a constant $C\in\C^{\ast}$ such that for all $\mathfrak{f}(\pi_B)$-allowable moduli $\m$ and all characters $\chi\colon \G\to \C^{\ast}$ of exact conductor $\m$ with $\chi_\infty=\epsilon$ we have
\begin{align*}
\chi(\mathfrak{L}_{\mathfrak{m}}(L/F,\pi_B)^{\epsilon})=
C\ &\frac{[U_T:U_T(\m)]^{2}}{N(\m)} L_{S(\chi)}(1,\eta) L_{S(\pi)\cup S(\chi)}(1,\eta) L_{S(\pi)\cap S(\chi)}(1,1_F)\\
&\times \prod_{\p \in S(\pi)\cap S(\chi)^{c}} \hspace{-1.5em}e_{\p}(E_\p/F_\p)
\cdot \frac{L^{\Sigma(\pi,\chi)}(1/2,\pi_E\otimes\chi)}{L^{\Sigma(\pi,\chi)}(1,\pi,\Ad)},
\end{align*}
if $\chi$ fulfills the Saito-Tunnell condition and
\begin{align*}
\chi(\mathfrak{L}_{\mathfrak{m}}(L/F,\pi_B)^{\epsilon})= 0
\end{align*}
if $\chi$ does not fulfill the Saito-Tunnell condition.
\end{Theorem}

\begin{proof}
The group cohomology of a discrete group is naturally isomorphic to the singular cohomology of its associated classifying space.
If the classifying space is a manifold, its singular cohomology with complex coefficients is isomorphic to its de Rham cohomology.
Thus, by invoking Shapiro`s Lemma we get an isomorphism
$$\ES\colon \mathcal{M}(K;\C)^{\epsilon}\too \HH_{\dR}^{d}(G(F)\backslash (G(\A^{\infty})/K\times \mathbb{H}^{d}),\C(\epsilon)).$$
By Matsushima's formula the space $\HH_{\dR}^{d}(G(F)\backslash (G(\A^{\infty})/K\times \mathbb{H}^{d}),\C(\epsilon))$ is generated by cohomological automorphic forms.
In particular, the image of $\kappa^{\pi_B,\epsilon}$ under $\ES$ is the differential form associated with a global (cohomological) new vector $\Phi$ of $\pi_B$.
The fact that the above identifications of cohomology groups behave well under pullback and cup products together with Remark \ref{fundamentaltori} implies that
$$\chi(\Theta_{\mathfrak{m}}(L/F,\pi_B)^{\epsilon})=[U_T:U_T(\m)] P_B(g_\m.\Phi,\chi)$$
holds up to multiplication by a non-zero constant, which is independent of $\chi$ and $\m$.
Here $g_\m\in G(\A^{\infty})$ is the element chosen at the end of Section \ref{Pullback} and
$$P_B(\phi,\chi)=\int_{T(F)\backslash T(\A)}\phi(t)\chi(t)\ dt$$
denotes the global toric period integral of $\phi\in\pi$.
Therefore, we also get the formula
$$\chi(\mathfrak{L}_{\mathfrak{m}}(L/F,\pi_B)^{\epsilon})=[U_T:U_T(\m)]^{2} |P_B(g_\m.\Phi,\chi)|^{2}$$
up to multiplication with a non-zero constant.

The second assertion follows from the vanishing criterion of toric periods integrals by Saito and Tunnell (see \cite{Sa} and \cite{Tu}).
Since $g_\m.\Phi$ is a test vector in the sense of \cite{FMP}, $\S$ 7.1, the first assertion follows from the main theorem of \textit{loc.cit.}
\end{proof}

\section{L-Invariants}
In Section 3.7 of \cite{Sp}, Spie\ss~constructs extensions of the Steinberg representation associated to characters of the multiplicative group of a $p$-adic field.
Such extensions were already constructed by Breuil in \cite{Br} in case the character under consideration is a branch of the $p$-adic logarithm.
After introducing a slightly improved version of Spie\ss' construction we will use it to give formulas for the leading term of anticyclotomic Stickelberger elements in the analytic rank zero situation.
Finally, we will relate the extension classes to a class coming from the $p$-adic upper half plane. This in turn allows us to recast the uniformization of Jacobians of certain Mumford curves purely in representation theoretic terms and proof the equality of arithmetic and automorphic $\LI$-invariants in certain cases.

\subsection{Extensions of the Steinberg representation} \label{Extension}
Let us fix a finite place $\p$ of $F$ which is split in $E$.
Further, let $R$ be a ring and $N$ a prodiscrete $R$-module.
We define the $N$-valued Steinberg representation by $$\St_\p(N)=C(\Ends,N)/N.$$
A continuous homomorphism $f\colon N\to N^{\prime}$ between prodiscrete $R$-modules induces a homomorphism $$f_\ast\colon \St_\p(N)\too \St_\p(N^{\prime}).$$
The canonical map $$\St_\p\otimes N\too\St_\p(N)$$ is an isomorphism if $N$ is discrete.
In this case, the map \eqref{deltasp} induces a $T_\p$-equivariant isomorphism
\begin{align*}
 \delta_{\p,N}\colon C_{c}^{0}(F_\p,N)\too \St_\p(N).
\end{align*}
 
Let $U_\p$ be the unipotent radical of $\Stab_{B_\p^{\ast}} (o_{\P^\tau})$, i.e.~we have $\Stab_{B_\p^{\ast}} (o_{\P^\tau})= E_\p^{\ast}\ U_\p$.
As before, the choice of the prime $\P$ lying above $\p$ gives rise to an identification $E_\p^{\ast}\cong F_{\p}^{\ast}\times F_{\p}^{\ast}$.
For a continuous homomorphism $l_\p \colon F_\p^{\ast} \to N$ we define $\widetilde{\mathcal{E}}(l_\p)$ as the set of pairs $(\varphi,y) \in C(B_\p^{\ast},N) \times R$ with
\begin{align*}
 \varphi \left( g u (t_1,t_2)\right)
		= \varphi(g) + y\cdot l_\p(t_1)
\end{align*}
for all $g \in B_\p^{\ast}$, $u \in U_\p$ and $(t_1,t_2) \in F_{\p}^{\ast}\times F_{\p}^{\ast}\cong E_\p^{\ast}$.
The group $B_\p^{\ast}$ acts on $\widetilde{\mathcal{E}}(l_\p)$ via $$g.(\varphi(h),y) = (\varphi(g^{-1}  h),y).$$
The subspace $\widetilde{\mathcal{E}}(l_\p)_0$ of tuples of the type $(\varphi,0)$ with constant $\varphi$ is $B_\p^{\ast}$-invariant.
Hence, we get an induced action of $G_\p$ on the quotient $\mathcal{E}(l_\p)=\widetilde{\mathcal{E}}(l_\p) /\widetilde{\mathcal{E}}(l_\p)_0$.

\begin{Lemma} \label{SteinbergExt}
\begin{enumerate}[(i)]\thmenumhspace
 \item Let $\pi \colon G_\p \to \Ends$ be the projection given by $g \mapsto g[o_{\P^\tau}]$.
       The following sequence of $R[G_\p]$-modules is exact:
       \begin{align*}
	0 \too \St_\p(N) \xrightarrow{(\pi^{\ast},0)} \mathcal{E}(l_\p) \xrightarrow{(0,\id_R)} R \too 0
       \end{align*}
       We define $b_{l_\p}$ to be the associated cohomology class in $\HH^1(G_\p,\St_\p(N))$.
 \item\label{SE2} For every continuous homomorphism $f\colon N\to N^{\prime}$ between prodiscrete R-modules the equality $$b_{f\circ l_\p}=f_{\ast}(b_{l_\p})$$ holds.
 \item Suppose that $N$ is discrete. Then, for the cohomology class $c_{l_\p}$ defined in \eqref{zchip} we have
$$\delta_{\p,N}^\ast(b_{l_\p}) = c_{l_\p}.$$
\end{enumerate}
\end{Lemma}

\begin{proof}
Parts (i) and (iii) are essentially proven in Lemma 3.11 of \cite{Sp}. \\
For the proof of (ii) let $f_{\ast}(\mathcal{E}(l_\p))$ be the pushout of the following diagram:
\begin{center}
 \begin{tikzpicture}
    \path 	(0,0) 	 node[name=A]{$\St_{\p}(N)$}
		(3,0) 	 node[name=B]{$\mathcal{E}(l_\p)$}
		(0,-1.5) node[name=C]{$\St_{\p}(N^{\prime})$};
    \draw[->] (A) -- (B) node[midway, above]{$(\pi^{\ast},0)$};
    \draw[->] (A) -- (C) node[midway, left]{$f_\ast$};
  \end{tikzpicture} 
\end{center}
The homomorphism $$\widetilde{\mathcal{E}}(l_\p)\too \widetilde{\mathcal{E}}(f\circ l_\p),\ (\varphi,y)\mapstoo (f\circ \varphi,y)$$ induces a map from $f_{\ast}(\mathcal{E}(l_\p))$ to $\mathcal{E}(f\circ l_\p)$.
Hence, they yield isomorphic extensions.
\end{proof}

\begin{Remark}
 Note that we get rid of the factor 2 showing up in Lemma 3.11 of \cite{Sp}, i.e.~the extension class constructed above is "one half" of the extension class constructed in \textit{loc.cit}.
\end{Remark}

For every prodiscrete $R$-module $N$ and $\p\in S_{\St}$ the integration pairing 
$$\Hom(\St_\p,R)\otimes \St_\p(N)\too N,$$
which was defined in \eqref{integration}, induces a cup product pairing
\begin{align}\begin{split} \label{cupproduct}
 \mathcal{M}& (\mathfrak{f}(\pi_B),S_\St,S_\tw;R)^\epsilon \otimes \HH^{1}(G(F),\St_{\p}(N)) \\
	     & \xlongrightarrow{\cup} \HH^{d+1}(G(F),\mathcal{A}(\mathfrak{f}(\pi_B),S_\St - \{\p\},S_\tw;N)^{\{\p\}}(\epsilon)).
	     \end{split}
\end{align} 
As a direct consequence of Lemma \ref{SteinbergExt} (iii) we get
\begin{Corollary} \label{diagram}
 Let $R$ be a ring, $N$ an $R$-module and $l\colon T_\p\too N$ a locally constant character.
 For $\p \in S_\St$ the following diagram is commutative:
\begin{center}
 \begin{tikzpicture}
    \path 	(0,0) 	 node[name=A]{$\mathcal{M}(\mathfrak{f}(\pi_B),S_\St,S_\tw;R)^\epsilon$}
		(6,0) 	 node[name=B]{$\HH^{d+1}(G(F),\mathcal{A}(\mathfrak{f}(\pi_B),S_\St-\{\p\},S_\tw;N)^{\{\p\}}(\epsilon))$}
		(0,-1.5) node[name=C]{$\HH^d(T(F),\mathcal{D}(\m,S;R)^{\infty}(\epsilon))$}
		(6,-1.5) node[name=D]{$\HH^{d+1}(T(F),\mathcal{D}(\m,S - \{\p\};N)^{\{\p\},\infty}(\epsilon))$};
    \draw[->] (A) -- (B) node[midway, above]{$\cup b_{l_\p}$};
    \draw[->] (A) -- (C) node[midway, left]{\small $\Delta_{\m,S_\St,S_\tw}$};
    \draw[->] (B) -- (D) node[midway, right]{\small $\Delta_{\m,S_\St-\{\p\},S_\tw}^{\{\p\}}$};
    \draw[->] (C) -- (D) node[midway, above]{$\cup c_{l_\p}$};
  \end{tikzpicture} 
\end{center} 
\end{Corollary}
\subsection{Leading terms} \label{leadingterms}
In this section we will compute leading terms of anticyclotomic Stickelberger elements.
Let $\pi$ be an automorphic representation as in Section \ref{JacquetLanglands} and let $S$ be the set of finite places $\p$ of $F$ which are disjoint from $\ram(B)$, split in $E$ and such that the local component $\pi_\p$ is the Steinberg representation.
To ease the notation, we are going to write $\mathcal{M}(\mathfrak{f}(\pi_B),S;N)$ instead of $\mathcal{M}(\mathfrak{f}(\pi_B),S,\emptyset;N)$ etc.
For simplicity we assume that $R_\pi$ is a principal ideal domain.

For a given anticyclotomic extension $L/E$ with Galois group $\G$ we denote by $I_\G$ the augmentation ideal of $R_\pi[\G]$, i.e.~the kernel of the projection $R_\pi[\G] \onto R_\pi$.
For $\p\in S$ we denote the local reciprocity map by $\rec_\p$, i.e.~
$$\rec_\p\colon T_\p\hooklongrightarrow T(\A)\xlongrightarrow{\rec_{L/E}}\G.$$
We also consider the homomorphism 
\begin{align}\label{univ}
\univ_\p \colon T_\p \to F_\p^\ast \otimes R_\pi
\end{align}
given by composing the isomorphism $T_\p \cong F_\p^\ast$ with the inclusion of $F_\p^\ast$ into $F_\p^\ast \otimes R_\pi$.
Given an $R_\pi$-module $N$ and a subset $\mathfrak{S}\subseteq S$ we let
$$\mathcal{M}(\mathfrak{f}(\pi_B),\mathfrak{S};N)^{\epsilon,\pi}\subseteq \mathcal{M}(\mathfrak{f}(\pi_B),\mathfrak{S};N)^{\epsilon}$$
be the the submodule on which $\mathbb{T}_\p$ acts via $\lambda_\p$ for all $\p\notin \mathfrak{S}\cup\ram(B)$.
Exactly as in Lemma 6.2 of \cite{Sp}, one can prove that for every $\p\in \mathfrak{S}$ the map $$\cup b_{\ord_{p}}\colon \mathcal{M}(\mathfrak{f}(\pi_B),\mathfrak{S};R_\pi)^{\epsilon,\pi}\too\HH^{d+1}(G(F),\mathcal{A}(\mathfrak{f}(\pi_B),\mathfrak{S}-\{\p\};R_\pi)^{\{\p\}}(\epsilon))^\pi$$ has finite cokernel and that both modules are free of rank one modulo torsion.

By a theorem of Borel and Serre (cf.~\cite{BS2}) $S_\p$-arithmetic groups are of type (VFL).
It follows that $H^{\ast}(\Gamma,N)$ is finitely generated if $N$ is a finitely generated $R_\pi$-module and that the functor $N\to H^{\ast}(\Gamma,N)$ commutes with direct limits.
It follows that the canonical map
\begin{align*}
 \omega_{\mathfrak{S},\p} \colon \HH^{d+1}(G(F),\mathcal{A} & (\mathfrak{f}(\pi_B),\mathfrak{S}-\{\p\};R_\pi)^{\{\p\}}(\epsilon))^\pi \otimes_{R_\pi} (F_\p^\ast \otimes R_\pi) \\ 
					     & \too \HH^{d+1}(G(F),\mathcal{A}(\mathfrak{f}(\pi_B),\mathfrak{S}-\{\p\};F_\p^\ast \otimes R_\pi)^{\{\p\}}(\epsilon))^\pi 
\end{align*}
has finite kernel and cokernel.
Let $\kappa_{\mathfrak{S}}$ be a generator of the maximal torsion-free quotient of $\mathcal{M}(\mathfrak{f}(\pi_B),\mathfrak{S};R_\pi)^{\epsilon,\pi}$.
For $\p\in \mathfrak{S}$ we define $n_{\mathfrak{S},\p}$ to be the lowest common multiple of the exponents of
\begin{itemize}
\item the cokernel of $\omega_{\mathfrak{S},\p}$ and
\item the torsion submodule of $\HH^{d+1}(G(F),\mathcal{A}(\mathfrak{f}(\pi_B),\mathfrak{S}-\{\p\};R_\pi)^{\{\p\}}(\epsilon))^{\pi}$.
\end{itemize}
Let $\gamma_{\mathfrak{S},\p}$ be the order of the cokernel of the homomorphism
\begin{align*}
 n_{\mathfrak{S},\p} \mathcal{M}(\mathfrak{f}(\pi_B),\mathfrak{S};R_\pi)^{\epsilon,\pi}
	\xlongrightarrow{\cup b_{\ord_\p}}
	n_{\mathfrak{S},\p} \HH^{d+1}(G(F),\mathcal{A}(\mathfrak{f}(\pi_B),\mathfrak{S}-\{\p\};R_\pi)^{\{\p\}}(\epsilon))^{\pi}.
\end{align*}
\begin{Def}
An element $q_{\mathfrak{S},\p} \in F_\p^\ast \otimes R_\pi$ fulfilling
\begin{align} \label{EQ}
n_{\mathfrak{S},\p} ((\kappa_{\mathfrak{S}} \cup  b_{\ord_\p}) \otimes q_{\mathfrak{S},\p})=  n_{\mathfrak{S},\p}(\gamma_{\mathfrak{S},\p}\cdot\kappa_{\mathfrak{S}} \cup b_{\univ_\p})
\end{align}
is called automorphic period of $\pi$ at $\p$ (with respect to $\mathfrak{S}$).
If $\mathfrak{S}=\{\p\}$ we simply write $q_\p=q_{\{\p\},\p}$.
\end{Def}
From the discussion above it follows that automorphic periods exist and are at least unique up to torsion.
It is easy to see that the $\Q_\pi$-vector subspace generated by $q_{\mathfrak{S},\p}$ in $F_\p^\ast \otimes \Q_\pi$ is independent of $\mathfrak{S}$.

\begin{Theorem}[Leading term]\label{leading}
For every anticyclotomic extension $L/E$ and every $\mathfrak{f}(\pi_B)$-allowable modulus $\m$ that bounds the ramification of $L/E$, the following equality holds in $I_\G^{|S_\m|}/I_\G^{|S_\m|+1}$ (up to sign):
\begin{align*}
 &n_{S_{\m}} \left(\prod_{\p\in S_\m} n_{S_\m,\p}\ord_{\p}(q_{S_\m,\p})\right) \Theta_\m(L/F,\pi_B)^{\epsilon} \\
=&n_{S_{\m}} \left(\prod_{\p\in S_\m} n_{S_\m,\p}\left(\rec_{\p}(q_{S_\m,\p})-1\right)\right) \Theta_{\m^S}(E/F,\pi_B)^{\epsilon}
\end{align*}
Here $\m^S$ denotes the maximal divisor of $\m$, which is coprime to $S$ and $n_{S_\m}$ is the non-zero integer defined in the discussion before Theorem \ref{ordvanish}.
\end{Theorem}

\begin{proof}
As in the proof of Lemma \ref{vanish2} there exists $\kappa' \in \mathcal{M}(\mathfrak{f}(\pi_B),S_\m;R_\pi)^\epsilon$ with
$$n_{S_\m} \Theta_\m(L/F,\pi_B)^{\epsilon} = \Delta_{\m,S_\m}(\kappa') \cap c_L(\m,S_\m,\epsilon).$$
If we apply Proposition \ref{vanish1} with $\mathfrak{a}_{\p} =\mathfrak{a} = I_\G$, we get $$ n_{S_\m} \Theta_\m(L/F,\pi_B)^{\epsilon} = \Delta_{\m,S_\m}(\kappa') \cap ((c_{\drec_{\p_1}} \cup \dots \cup c_{\drec_{\p_s}}) \cap \overline{c_L}(\m,S_\m,\epsilon)),$$ where $S_\m=\{\p_1,\ldots,\p_s\}$.
By Corollary \ref{diagram} we obtain
\begin{align*}
 &n_{S_\m} \Theta_\m(L/F,\pi_B)^{\epsilon}\\
= &\Delta^{\{\p_i\}}_{\m,S_\m - \{ \p_i\}}(\kappa' \cup b_{\drec_{\p_i}}) 
							       \cap ((c_{\drec_{\p_1}} \cup \dots \cup \widehat{c_{\drec_{\p_i}}} \cup \dots \cup c_{\drec_{\p_s}}) \cap \overline{c_L}(\m,S_\m,\epsilon))
\end{align*}
for every $i \in \{1,\dots,s\}$.

By Lemma \ref{SteinbergExt} \eqref{SE2} the following diagram is commutative for every $\p\in S_\m$:
\begin{center}
 \begin{tikzpicture} 
    \path 	(0,0)  node[name=A]{$\mathcal{M}(\mathfrak{f}(\pi_B),S_\m;R_\pi)^\epsilon$}
		(7,0)  node[name=B]{$\HH^{d+1}(G(F),\mathcal{A}(\mathfrak{f}(\pi_B),S_\m-\{\p\};F_\p^\ast \otimes R_\pi)^{\{\p\}}(\epsilon))$}
		(7,-2) node[name=C]{$\HH^{d+1}(G(F),\mathcal{A}(\mathfrak{f}(\pi_B),S_\m-\{\p\};I_\p/I_\p^2)^{\{\p\}}(\epsilon))$}
		(7,2)  node[name=D]{$\HH^{d+1}(G(F),\mathcal{A}(\mathfrak{f}(\pi_B),S_\m-\{\p\};R_\pi)^{\{\p\}}(\epsilon))$};
    \draw[->] (A) -- (B) node[midway, above]{$\cup b_{\univ_\p}$};
    \draw[->] (A) -- (C) node[midway, below left]{$\cup b_{\drec_\p}$};
    \draw[->] (A) -- (D) node[midway, above left]{$\cup b_{\ord_\p}$};
    \draw[->] (B) -- (C) node[midway, right]{$(\drec_\p)_\ast$};
    \draw[->] (B) -- (D) node[midway, right]{$(\ord_\p)_\ast$};
  \end{tikzpicture} 
\end{center}
Applying $\drec_\p$ to \eqref{EQ} and using the commutativity of the lower triangle of the diagram we get
\begin{align*}
 n_{S_\m,\p} (\rec_\p(q_{S_\m,\p})-1) \kappa' \cup b_{\ord_\p}  = n_{S_\m,\p} \gamma_{S_\m,\p} \kappa' \cup b_{\drec_\p}.
\end{align*}
By the commutativity of the upper triangle of the diagram we see that
$$\gamma_{S_\m,\p} = \ord_\p(q_{S_\m,\p}).$$
Hence, it is enough to show the following lemma.
\end{proof}
\begin{Lemma}
The equality
$$(\Delta_{\m,S_\m}(\kappa') \cup c_{\ord_{\p_1}} \cup \dots \cup c_{\ord_{\p_s}}) \cap \overline{c_E}(\m,S_\m,\epsilon)=\pm  n_{S_\m}\ \Theta_{\m^S}(E/F,\pi_B)^{\epsilon}$$
holds in $R_{\pi}$.
\end{Lemma}
\begin{proof}
By Remark \ref{loccomp} we have $c_{\ord_{\p_i}}=(\alpha_{\p_i})_{\ast}(c_{\p_i})$ for all $1\leq i\leq s$. Thus, we get
\begin{align*}
&(c_{\ord_{\p_1}} \cup \dots \cup c_{\ord_{\p_s}}) \cap \overline{c_E}(\m,S_\m,\epsilon)\\
=& (\alpha_{\p_1}\otimes\ldots\otimes\alpha_{\p_s})_{\ast}((c_{\p_1} \cup \dots \cup c_{\p_s}) \cap \overline{c_E}(\m,S_\m,\epsilon))\\
=&\pm  (\alpha_{\p_1}\otimes\ldots\otimes\alpha_{\p_s})_{\ast} (c_E(\m^S,\emptyset,\epsilon)).
\end{align*}
The second equality holds by Lemma \ref{fundclasses}. By Lemma \ref{comp} \eqref{comp2} we have 
\begin{align*}
&\Delta_{\m,S}(\kappa')\cap (\alpha_{\p_1}\otimes\ldots\otimes\alpha_{\p_s})_{\ast}\cap c_E(\m^S,\emptyset,\epsilon)\\
=& (\alpha_{\p_1}\otimes\ldots\otimes\alpha_{\p_s})^{\ast}(\Delta_{\m,S}(\kappa')) \cap c_E(\m^S,\emptyset,\epsilon)\\
=& \Delta_{\m^S}(n_{S_\m} \kappa)\cap c_E(\m^S,\emptyset,\epsilon)\\
=& n_{S_\m} \Theta_{\m^S}(E/F,\pi_B)^{\epsilon}
\end{align*}
and thus, the claim follows.
\end{proof}

\begin{Remark}
\begin{enumerate}[(i)]\thmenumhspace
\item Suppose we are in the CM case, i.e.~$d=0$, and that $S_\m=\{\p\}$. Then the module $$\HH^{1}(G(F),\mathcal{A}(\mathfrak{f}(\pi_B),\emptyset;R_\pi)^{\{\p\}}(\epsilon))^{\pi}$$ is torsion-free. Thus, $n_{S_\m,\p}$ is just the exponent of the cokernel of $\omega_{S_\m,\p}$.
\item Using the norm relations and the interpolation formulae one can determine $\Theta_{\m^S}(E/F,\pi_B)^{\epsilon}$ explicitly in terms of the special value at $1/2$ of the untwisted $L$-function $L(s,\pi_E)$.
\item From the proof of Theorem \ref{leading} we see that $\ord_{\p}(q_\p)$ is non-zero for every automorphic period $q_\p$.
\end{enumerate}
\end{Remark}
\subsection{Jacobians of Mumford curves}\label{Mumford}
Let $\p$ be a finite place of $F$. In this section we give a representation theoretic formulation of the $p$-adic uniformization of Jacobians of certain Mumford curves. 
Let $\PPP$ be the algebraic variety $G_\p/\Stab_{G_\p} (o_{\P^\tau})$ and $\C_p$ the completion of an algebraic closure of $F_\p$. Let $\sigma_\p\colon F_\p^{\ast}\to \C_p^{\ast}$ denote the natural embedding.
Let $\mathcal{H}_p$ be the rigid analytic space
$$\mathcal{H}_\p(\C_p)=\PPP(\C_p)-\PPP(F_\p).$$
The group $G_\p$ acts on $\mathcal{H}_\p(\C_p)$ and therefore, by linear extension, on the space $\Div(\mathcal{H}_\p(\C_p))$ of (naive) divisors on $\mathcal{H}_\p(\C_p)$, i.e.~the space of formal $\Z$-linear combinations of points of $\mathcal{H}_\p(\C_p)$.
The subspace $\Div^{0}(\mathcal{H}_\p(\C_p))$ of divisors of degree $0$ is $G_\p$-invariant.
Sending a degree $0$ divisor $D$ to the unique (up to multiplication by a constant) rational function with divisor $D$, yields a $G_\p$-equivariant homomorphism
$$\Psi\colon\Div^{0}(\mathcal{H}_\p(\C_p))\mapstoo C(\PPP(F_\p),\C_p^{\ast})/\C_\p^{\ast}=\St_\p(\C_p^{\ast}).$$

\begin{Lemma}\label{comparison} Let $b_{\adic_{\p}}\in \HH^1(G_\p,\Div^{0}(\mathcal{H}_\p(\C_p)))$ be the class of the extension
\begin{align}\label{upperhalfplane}
 0\too \Div^0(\mathcal{H}_\p(\C_p))\too \Div(\mathcal{H}_\p(\C_p))\xlongrightarrow{\deg}\Z\too 0
\end{align} and $b_{\univ_{\p}}\in\HH^1(G_\p,\St_\p(F_\p^{\ast}))$ the universal class induced by \eqref{univ} (with $R_\pi=\Z$).
The equality $\Psi_{\ast}(b_{\adic_{\p}})=(\sigma_\p)_\ast (b_{\univ_{\p}})$ holds.
\end{Lemma}

\begin{proof}
Let us a choose an isomorphism $\Xi\colon B_\p^{\ast}\xrightarrow{\cong} \GL_2(F_\p)$ such that $$\Xi(t_1,t_2)=\begin{pmatrix} t_1 & 0 \\ 0 & t_2 \end{pmatrix}$$
holds for all $(t_1,t_2)\in E_{\p}^{\ast}$ and $\Xi(\Stab_{B_\p^{\ast}} (o_{\P^\tau}))$ is equal to the group of invertible upper triangular matrices.
Thus, there is also an induced isomorphism between $\PPP$ and $\PP_{F_\p}$.
For $z\in\mathcal{H}_\p(\C_p)$ we define the function $$\Phi_z\colon \GL_2(F_\p)\too \C_p^{\ast},\quad \begin{pmatrix} a & b \\ c & d \end{pmatrix} \mapstoo cz-a.$$
A short calculation shows that the homomorphism $$\widetilde{\Psi}\colon \Div(\mathcal{H}_\p(\C_p))\too \mathcal{E}(\sigma_\p),\quad \sum_z n_z[z]\mapstoo (\prod_z (\Phi_z)^{n_z},\sum_z n_z)$$ is $G_\p$-equivariant.
One immediately checks that the diagram
\begin{center}
 \begin{tikzpicture}
    \path 	(0.5,0) node[name=A]{$0$}
		(3,0) node[name=B]{$\Div^0(\mathcal{H}_\p(\C_p))$}
		(6,0) node[name=C]{$\Div(\mathcal{H}_\p(\C_p))$}
		(8.5,0) node[name=D]{$\Z$}
		(10.5,0) node[name=E]{$0$}
		(0.5,-2) node[name=F]{$0$}
		(3,-2) node[name=G]{$\St_\p(\C_p^{\ast})$}
		(6,-2) node[name=H]{$\mathcal{E}(\sigma_\p)$}
		(8.5,-2) node[name=I]{$\Z$}
		(10.5,-2) node[name=J]{$0$};
    \draw[->] (A) -- (B) ;
		\draw[->] (B) -- (C) ;
		\draw[->] (C) -- (D) ;
		\draw[->] (D) -- (E) ;
		\draw[->] (F) -- (G) ;
		\draw[->] (G) -- (H) ;
		\draw[->] (H) -- (I) ;
		\draw[->] (I) -- (J) ;
    \draw[->] (B) -- (G) node[midway, right]{$\Psi$};
    \draw[->] (C) -- (H) node[midway, right]{$\widetilde{\Psi}$};
    \draw[->] (D) -- (I) node[midway, right]{$=$};
  \end{tikzpicture} 
\end{center}
is commutative and therefore, the claim follows.
\end{proof}

A discrete subgroup $\Gamma \subseteq G_\p$ is called a ($p$-adic) Schottky group if it is finitely generated and torsion-free.
Since the stabilizers of vertices and edges of the Bruhat-Tits tree $\T_\p$ are compact subgroups of $G_\p$ it follows that Schottky groups act freely on $\T_\p$. 
A group acting freely on a tree is free. (See for example Section 3.3 of \cite{Trees})
Hence, we have that for a Schottky group $\Gamma$ and an abelian group $A$ the canonical map
\begin{align}\label{basechange}
\HH^{1}(\Gamma,\Z)\otimes A\too\HH^{1}(\Gamma,A)
\end{align}
is an isomorphism.
We consider the algebraic torus $\mathbb{G}_\Gamma=\HH^{1}(\Gamma,\Z)\otimes\mathbb{G}_{m}$.
It follows from \eqref{basechange} that we have a canonical isomorphism $\mathbb{G}_{\Gamma}(\C_p)\cong\HH^{1}(\Gamma, \C_p^{\ast}).$
The cup product pairing \eqref{cupproduct} induces a homomorphism $$\int_{\univ}\colon\HH^{0}(\Gamma,\St_\p(\Z)^{\vee})\xrightarrow{\cup b_{\univ_\p}}\HH^{1}(\Gamma,\C_p^{\ast})=\mathbb{G}_\Gamma(\C_p).$$ We define $L_\Gamma$ to be its image.

Let $X$ be a smooth proper curve over $F_\p$ with Jacobian $\Jac_X$.
We assume that $X$ admits a uniformization by the $p$-adic upper half plane, i.e.~there exists a Schottky group $\Gamma_X$ and a $\Gal(\C_p/F_\p)$-equivariant rigid analytic isomorphism $$X(\C_p)\cong \Gamma_X\backslash\mathcal{H}_\p(\C_p).$$

\begin{Theorem}[$p$-adic uniformization]\label{unif}
There is a $\Gal(\C_{p}/F_\p)$-equivariant rigid analytic isomorphism
$$\Jac_X(\C_p)\cong \mathbb{G}_{\Gamma_X}(\C_p)/L_{\Gamma_X}.$$
\end{Theorem}

\begin{proof}
Let $$\delta\colon\HH_1(\Gamma_X,\Z)\too\HH_0(\Gamma_X,\Div^0(\mathcal{H}_\p(\C_p)))$$ be the boundary map coming from the short exact sequence \eqref{upperhalfplane}.
In Section 2 of \cite{Das05} Dasgupta defines a so-called multiplicative integral
$$\Xint{\times}\colon\HH_0(\Gamma_X,\Div^0(\mathcal{H}_\p(\C_p)))\too \mathbb{G}_{\Gamma_X}(\C_p).$$
We write $\widetilde{L}_{\Gamma_X}$ for be the image of the map $\Xint{\times} \circ\delta$.
By Dasgupta's variant of the Manin-Drinfeld Theorem  (see \cite{Das05}, Theorem 2.5) there is a rigid analytic isomorphism
$$\Jac_X(\C_p)\cong \mathbb{G}_{\Gamma_X}(\C_p)/\widetilde{L}_{\Gamma_X}.$$
It follows from Lemma \ref{comparison} that the diagram 
\begin{center}
 \begin{tikzpicture}
    \path 	(0,0) node[name=A]{$\HH_1(\Gamma_X,\Z)$}
		(4,0) node[name=B]{$\HH_0(\Gamma_X,\Div^0(\mathcal{H}_\p(\C_p)))$}
		(8,0) node[name=C]{$\mathbb{G}_{\Gamma_X}(\C_p)$}
		(0,-2) node[name=F]{$\HH_1(\Gamma_X,\Z)$}
		(4,-2) node[name=G]{$\HH_0(\Gamma_X,\St_\p(\C_p^{\ast}))$}
		(8,-2) node[name=H]{$\mathbb{G}_{\Gamma_X}(\C_p)$};
    \draw[->] (A) -- (B) node[midway, above]{$\delta$};
		\draw[->] (B) -- (C) node[midway, above]{$\Xint{\times}$};
		\draw[->] (F) -- (G) node[midway, above]{$\delta_{\univ}$};
		\draw[->] (G) -- (H) node[midway, above]{$\Xint{\times}$};
    \draw[->] (A) -- (F) node[midway, right]{$=$};
    \draw[->] (B) -- (G) node[midway, right]{$\Psi_\ast$};
    \draw[->] (C) -- (H) node[midway, right]{$=$};
  \end{tikzpicture}
\end{center}
is commutative. Here $\delta_{\univ}$ denotes the boundary map coming from the short exact sequence corresponding to $b_{\univ_\p}$.

Since $\Gamma_X$ acts freely on the Bruhat-Tits tree $\T_\p$ of $G_\p$ it follows that $\St_\p(\Z)=\HH_{c}^{1}(\T_\p,\Z)$ is a dualizing module for $\Gamma_X$.
In particular, $\HH_{1}(\Gamma_X,\St_\p(\Z))$ is a free $\Z$-module of rank one and taking cap product with a generator induces an isomorphism
$$\textnormal{D}\colon\HH^{0}(\Gamma_X,\St_\p(\Z)^{\vee})\too \HH_{1}(\Gamma_X,\Z).$$
Going through Dasgupta's construction one sees that
$$\int_{\univ}=\pm\Xint{\times}\circ\ \delta_{\univ}\circ \textnormal{D}$$
holds.
\end{proof}
\subsection{Comparison of $\LI$-invariants}\label{uniformization}
We apply the results of the previous section to Shimura curves in order to compare automorphic and algebraic periods.
From now on we assume that $d=0$, i.e.~that $E$ is totally imaginary.
In particular, the quaternion algebra $B$ is totally definite.
We fix a prime $\p\notin\ram(B)$ of $F$ such that $\pi_\p$ is the Steinberg representation and an arbitrary Archimedean place $v$ of $F$.
Let $\widetilde{B}$ be the quaternion algebra over $F$ that has the same invariants at all places away from $\p$ and $v$ but is split at $v$ and, therefore, non-split at $\p$.
There exists a Jacquet-Langlands lift $\pi_{\widetilde{B}}$ of $\pi$ to $\widetilde{B}$ of conductor $\mathfrak{f}(\pi_{\widetilde{B}})=\p^{-1}\mathfrak{f}(\pi_{B})$.
Let $X_{\widetilde{B},\mathfrak{f}(\pi_{\widetilde{B}})}$ be the associated (not necessarily connected) Shimura curve of level $\Gamma_0(\mathfrak{f}(\pi_{\widetilde{B}}))$.
We define $A_\pi$ to be the maximal quotient of the Jacobian $\Jac(X_{\widetilde{B},\mathfrak{f}(\pi_{\widetilde{B}})})$, on which the Hecke algebra acts via the character induced by $\pi_{\widetilde{B}}$.
For simplicity, we assume that the field of definition of $\pi$ is equal to $\Q$.
Hence, $A_\pi$ is an elliptic curve with split multiplicative reduction at $\p$.
Therefore, there exists a period $q_\p^{\Tate}\in F_\p^{\ast}$ and a rigid analytic isomorphism
$$\C_\p^{\ast}/q_\p^{\Tate}\xlongrightarrow{\cong}A_\pi(\C_\p).$$
By the Cerednik-Drinfeld Theorem (cf.~\cite{BZ}) every connected component of the Shimura curve has a uniformization by the $p$-adic upper half plane.
More precisely, the uniformization of the whole Shimura curve combined with Proposition \ref{unif} yields a Hecke-equivariant isogeny between $\Jac(X_{\widetilde{B},\mathfrak{f}(\pi_{\widetilde{B}})})(\C_\p)$ and the cokernel of the map
$$\mathcal{M}(\mathfrak{f}(\pi_B),\{\p\};\Z)\xlongrightarrow{\cup b_{\univ}}\HH^{1}(G(F),\mathcal{A}(\mathfrak{f}(\pi_B);\C_\p^{\ast})^{\{\p\}}).$$
Taking $\pi_B$-isotypical components we arrive at
\begin{Theorem}[Comparison of periods]\label{periods}
 Let $E$ be totally imaginary and let $\p\notin\ram(B)$ be a prime of $F$ such that $\pi_\p$ is the Steinberg representation.
 Then the lattices generated by the automorphic period $q_\p$ and the algebraic period $q_\p^{\Tate}$ are commensurable.
\end{Theorem}

\begin{Remark}
 One also expects that the analogue of Theorem \ref{periods} holds for $d>0$. See for example \cite{Gr}, Conjecture 2, for the case that the narrow class number of $F$ is one and Conjecture 4.8 of \cite{GMS} for the general case. 
 The equivalence of their formulations and ours follows from Lemma \ref{comparison}.
\end{Remark}

\bibliographystyle{abbrv}
\bibliography{bibfile}
\vfill
\end{document}